\documentclass[11pt]{article}
\usepackage[latin1]{inputenc}
\usepackage[dvips]{graphicx}
\usepackage{latexsym}
\usepackage{amsmath,amssymb,amsfonts,amsthm}
\usepackage{xcolor}

\newcommand{\nrn}{\rightarrow+\infty}
\newcommand{\xrn}{\xrightarrow}

\newcommand{\ER}{\mathbb {R}}\newcommand{\EN}{\mathbb {N}}
\newcommand{\PE}{\mathbb {P}}
\newcommand{\ES}{\mathbb{E}}

\newcommand{\guil}{\textquotedblleft}
\newcommand{\psg}{\langle }
\newcommand{\psd}{\rangle }
\newcommand{\bx}{\bar{X}}
\newcommand{\bz}{\Sigma}
\newcommand{\Gn}{{n\gamma}}
\newcommand{\Gnp}{{(n+1)\gamma}}
= 0 cm \evensidemargin = -0.3 cm \marginparwidth = 0 cm
\newtheorem{theorem}{ \textnormal{\bf{T\scriptsize{HEOREM}}}}
\newtheorem{prop}{\textnormal{\bf{P\scriptsize{ROPOSITION}}}}
\newtheorem{Corollaire}{\textnormal{\bf{C\scriptsize{OROLLARY}}}}
\newtheorem{lemme}{\textnormal{\bf{L\scriptsize{EMMA}}}}

\theoremstyle{definition}
\newtheorem{definition}{\textnormal{\bf{D}\scriptsize{EFINITION}}}
\theoremstyle{remark}

\newtheorem{Remarque}{\textnormal{\bf{R\scriptsize{EMARK}}}}

\author{Serge Cohen\footnote{ Institut de Mathématiques de Toulouse, Université de Toulouse, 118 route de Narbonne F-31062 Toulouse Cedex  9 E-mail: \texttt{serge.cohen@math.univ-toulouse.fr}},
Fabien Panloup\footnote{Laboratoire de Statistiques et Probabilités, Université de Toulouse$\&$INSA Toulouse, 135, Avenue de Rangueil, 31077 Toulouse Cedex 4, E-mail: \texttt{fabien.panloup@.math.univ-toulouse.fr}}
}
\title{\textbf{ Approximation of stationary solutions of Gaussian driven
    Stochastic Differential Equations}}
\begin{document}
\maketitle
\begin{abstract}	
We study sequences of empirical measures of Euler schemes  associated to some non-Markovian SDEs: SDEs driven by Gaussian processes with stationary increments. We obtain the functional convergence of this sequence to a stationary solution to the SDE. Then, we end the paper by some specific properties of this stationary solution. We show that, in contrast to Markovian SDEs, its initial random value and the driving Gaussian process are always dependent. However, under an integral representation assumption, we also obtain that the past of the solution is independent to the future of the underlying innovation process of the Gaussian driving process.
\end{abstract}

\noindent \textit{AMS classification (2000)}: 60G10, 60G15, 60H35.
\section{Introduction}
\noindent The study of steady state of dynamical systems  is very important for many experimental
sciences like Physics, Chemistry, or Biology, since very often
measure can only be obtained in that regime. In the Markovian 
setting the study of long time behavior and stationary solutions of dynamical 
systems is a classical domain of both Mathematic and Probability. 
Nevertheless in many situations the driving noise of the dynamical 
system has long range dependence properties and the solution is
not Markovian.  

In this paper, we deal   with an  $\ER^d$-valued process $(X_t)_{t\ge0}$
solution to the SDE 
of the following form:
\begin{equation}\label{fractionalSDE0}
dX_t=b(X_t)dt+ dZ_t
\end{equation}
where $(Z_t)_{t\ge0}$ is  a  continuous centered
Gaussian process with ergodic  stationary increments. For this class of SDEs, our principal aim is to approximate some stationary solutions under some mean-reverting assumptions on $b$ and 
weak assumptions on $(Z_t)_{t\ge0}$  including ergodicity of the discrete increments that will be made precise in the next section. Note that, since for any matrix $\sigma$, $(\tilde{Z}_t)_{t\ge0}=(\sigma Z_t)_{t\ge0}$ is also a continuous centered Gaussian process  with stationary ergodic increments, SDEs of type \eqref{fractionalSDE0} include the following ones: $dX_t= b(X_t)dt+\sigma dZ_t$. However, we can remark the main restriction: we do not consider the case where $\sigma$ is not constant. It allows us on one
hand to avoid technicalities related to stochastic integration and on
the other hand to generalize some results of~\cite{hairer}, when 
the driving noise is not a fractional Brownian motion. Please note that, when $b(x)=-  x$, the solution of
\eqref{fractionalSDE0} is an Orstein-Uhlenbeck type process, where the driving process
may be more general than 
a fractional Brownian motion (see \cite{cheridito} for a study of fractional Ornstein-Uhlenbeck processes). We obtain bounds for a discrete version of this generalized Ornstein-Uhlenbeck process, which are an important tool in our proofs and which may have interest
of their own (see Lemma \ref{lemme2}).  \\
In this work, our approach is quite different to that of ~\cite{hairer}. Actually, we choose to first approximate stationary solutions of an ergodic discrete model associated with \eqref{fractionalSDE0}. Then, stationary solutions of the SDE are exhibited as limits of these stationary solutions. More precisely, in a first step, we study a sequence of  functional empirical occupation measures of an Euler scheme $(\bar{X}_{\Gn})$ with step $\gamma>0$ associated with ~\eqref{fractionalSDE0} and show under some mean-reverting assumptions on $b$, that, when $n\rightarrow+\infty$, this sequence has almost surely ($a.s.$ later on) some weak convergence properties to the distribution of a stationary Euler scheme with step $\gamma$ of the SDE. 
Denoting these stationary solutions by $Y^{(\infty,\gamma)}$, we show in a second step, that $(Y^{(\infty,\gamma)})_{\gamma}$ is tight for the uniform convergence on compact sets and that its weak limits  (when $\gamma\rightarrow0$) are stationary solutions to \eqref{fractionalSDE0}.

For a Markovian SDE, this type of approach is used as a way of numerical approximation of the invariant distribution and more generally of the distribution of the Markov process when stationary (see \cite{talay}, \cite{LP1}, \cite{LP2}, \cite{lemaire2}, \cite{panloup1}, \cite{PP1}).  Here, even if the discrete model can  be simulated, we essentially use it as a natural way of construction of stationary solutions of the continuous model and the computation problems are out of the scope of this paper. \\
\noindent In Section \ref{mainresult}, we make the mathematical framework precise  and we state our main results 
of convergence   to the stationary regime of SDE \eqref{fractionalSDE0}. Then, Sections \ref{tightness1}, \ref{identification1} and \ref{identification2} are devoted to the proof of the main results. First, in Sections \ref{tightness1} and \ref{identification1}, we study  the long time behavior of the sequence ${(\bar{X}_{\Gn})}_{n \ge 1}$
 (when $\gamma$ is fixed) and the convergence properties (when $n\rightarrow+\infty$) of the sequence of functional empirical occupation measures of the continuous-time Euler scheme. We show that this sequence is $a.s.$ tight for the topology of uniform convergence on compact sets and that 
its weak limits are  stationary solutions to the ``discretized'' SDE. Second, in Section \ref{identification2}, we focus on the behavior of these weak limits when $\gamma\rightarrow0$. In Section \ref{sect-properties}, we give some properties of the stationary solution. We first show that the initial random value and the driving process are dependent as soon as the Gaussian process has dependent increments. However, assuming some integral  representation of  $Z$ (with respect to a white noise), we then prove that the past of the stationary solution that are built with our method, are independent to the future of the underlying innovation process of $Z$. Please note that a similar result is also proven in the discrete case for the stationary Euler scheme associated with the SDE. Section~\ref{appendix} is an Appendix where we obtain some control of the moment of the supremum of a Gaussian process 
and a technical Lemma showing 
that we can realize the stationary solutions with the help of an innovation representation 
of the driving process.

\section{Framework and main results}\label{mainresult}
Before outlining the sequel of the paper, we list some
 notations. 
 Throughout this paper, $ \ER_+=[0,\infty)$. We denote by ${\cal
 C}(\ER_+,\ER^d)$ (resp. $\mathbb{D}(\ER_+,\ER^d)$) the space of
 continuous (resp. càdlàg functions) endowed with the  uniform 
 convergence on compact sets (resp. Skorokhod (see $e.g.$ \cite{billingsley})) topology, and  by  ${\cal C}^k(\ER_+,\ER^d),$  the set of $k$th differentiable functions. 
 The Euclidean norm is denoted by $|\, .\,|$. For a measure $\mu$ and a $\mu$-measurable function f, we set $\mu(f)=\int fd\mu$.
 Finally, we will denote by $C$  every non explicit positive constant. In particular, it can change from line to line.\\
Let us first consider assumptions for the driving noise $ (Z_t)_{t\ge0}=(Z^1_t,\ldots,Z^d_t)_{t\ge0}:$ we assume that $(Z_t)_{t\ge0}$ is a centered Gaussian process satisfying $Z_0=0$  and, for every $i\in\{1,\ldots,d\}$, 
we denote by ${c}_i:\ER\rightarrow\ER_+$, the following function of $(Z^i_t)_{t\ge0}$: for every positive $s,t$,
$$\ES[(Z^i_t-Z^{i}_s)^2]={c}_i(t-s).$$
Note that ${c}_i(0)=0$. For every integer $n\ge0$, let us denote by $\Delta_n=Z_\Gn-Z_{(n-1)\gamma} $ when $\gamma > 0$ is fixed.
Setting $\phi_\gamma^i(n):=\ES[\Delta^i_1\Delta^i_{n+1}]$ for $i=1,\ldots,l$, we have:
\begin{equation}\label{covariance}
\phi_\gamma^i(n)=\frac{1}{2}\left[{c}_i((n+1)\gamma)-2{c}_i(n\gamma)+{c}_i((n-1)\gamma)\right].
\end{equation}
We denote by $(\bar{Z}_t)_{t\ge0}$ the ``discretized'' Gaussian process defined by $\bar{Z}_{n\gamma}:=Z_{n\gamma}$ for every $n\ge0$ and,
$$\bar{Z}_{t}=\bar{Z}_{n\gamma}\quad\forall t\in[n\gamma,(n+1)\gamma).$$
We introduce  assumption $\mathbf{(H_1)}$ on the functions ${c}_i$, $i\in\{1,\ldots,{d}\}$. More precisely, we impose some conditions on the second derivative of ${c}_i$ near 0 and $+\infty$ which correspond respectively to  some conditions on the local behavior and on the memory of the process. \\

\noindent $\mathbf{(H_1)}$ For every $i\in\{1,\ldots,{d}\}$, ${c}_i$ is continuous on $\ER_+$ and  ${\cal C}^2$ on $(0,+\infty)$. Moreover,  there exist $a_i\in(0,2)$ and $b_i>0$ such that:
\begin{equation}\label{cov-bounds}
|{c}_i''(t)|\le
\begin{cases}
C t^{-a_i}&\forall t\in(0,1)\\
C t^{-b_i}&\forall t\ge1.
\end{cases}
\end{equation}
Let us recall that for a fractional Brownian motion with Hurst index $H$, these assumptions are satisfied with $a_i=b_i=2-2H$. One can also check that  ~\eqref{cov-bounds} implies that in a neighborhood of $0$,
\begin{equation}\label{consequence-cov-bounds}
c_i(t)\le C\begin{cases} t &\textnormal{if $a_i\in(0,1)$,}\\
t\ln t&\textnormal{if $a_i=1$,}\\
t^{2-a_i}&\textnormal{if $a_i\in(1,2)$}.
\end{cases}
\end{equation}
In particular, the sample paths of $ (Z_t)_{t\ge0} $ are almost surely continuous.
Futhermore, we derive from assumption $\mathbf{(H_1)}$ that for every $i\in\{1,\ldots,{d}\}$,  
$\ES[\Delta_1^i\Delta_n^i]\rightarrow0$ as $n\rightarrow+\infty$. Then, it follows from \cite{cornfeld} that $(\Delta_n)_{n\ge1}$ is an ergodic sequence, $i.e.$ that for every $F:(\ER^d)^{\mathbb{N}}\rightarrow\ER$ such that $\ES[|F((\Delta_n)_{n\ge1})|]<+\infty$,
\begin{equation}\label{ergodicite-increments}
\frac{1}{n}\sum_{k=1}^n F((\Delta_k)_{k\ge n})\xrn{n\nrn}\ES[F((\Delta_n)_{n\ge1})].
\end{equation}

\noindent Let us now introduce  some stability assumptions $\mathbf{(H_2)}$ and $\mathbf{(H_3)}$  concerning the stochastic differential equation
\begin{equation}\label{fractionalSDE}
dX_t=b(X_t)dt+ dZ_t,
\end{equation}
where $b:\ER^d\rightarrow\ER^d$ is a continuous function.\\
\newpage
 \noindent $\mathbf{(H_2)}$:
 \begin{itemize}
 \item[(i)] There exists $C>0$ such that $|b(x)|\le C(1+|x|)\quad \forall x\in\ER^d.$
 \item[(ii)] There exist $\beta\in\ER$ and $\alpha>0$ such that 
 $$\psg x,b(x)\psd\le  \beta-\alpha |x|^2.$$
 \end{itemize}
\noindent  $\mathbf{(H_3)}$: $b$ is a Lipschitz continuous function and there
 exist $\alpha>0$ and $\beta \ge 0$, such that $\forall x,y\in\ER^d$,
 $$\psg b(x)-b(y),x-y\psd\le \beta-\alpha |x-y|^2.$$
 When $\mathbf{(H_3)}$ holds for $\beta=0$, we will denote it by $\mathbf{(H_{3,0})}.$ 
 \begin{Remarque} 
 The reader can check that $\mathbf{(H_3)}$ for  some  $\alpha > 0 , \beta\ge0 $ implies 
$\mathbf{(H_2)} (ii)$ for  some  $0 < \alpha' <\alpha $ and $\beta'\ge\beta$ by taking $y=0$ in $\mathbf{(H_3)}.$ 
As well, the fact that $b$ is Lipschitz continuous implies $\mathbf{(H_3)}(i).$ One may argue that the Lipschitz assumption 
is too strong for our purpose but we chose to keep this assumption for two reasons.
First, this assumption is only needed for the SDE in continuous time 
(i.e. the 2. of Theorem~\ref{principal}). Second it is a convenient way to have the sublinearity 
assumption $\mathbf{(H_2)}(i).$


 \end{Remarque}
\noindent When $b$ is a Lipschitz continuous function, it is obvious using  Picard iteration arguments that for any initial random variable $\xi$ 
a.s. finite there exists a unique  solution $(X_t)_{t\ge0}$ to~\eqref{fractionalSDE} such that $X_0=\xi $ which is adapted to the filtration $\sigma(\xi,Z_s, 0\le  s \le t).$ Then,
\begin{equation}\label{eq-idsbis}
 X_t= \xi + \int_0^t b(X_s) ds +  Z_t, \quad \forall t >0.
 \end{equation}
\noindent Please note that the integral in~\eqref{eq-idsbis} is always defined since the sample
paths of $ (X_t)_{t\ge0} $ and $ (Z_t)_{t\ge0} $ are continuous.  \\

\noindent Let us now define a stationary solution to~\eqref{fractionalSDE}.
\begin{definition} \label{stat-sol}
Let  $b:\ER^d\rightarrow\ER^d$ be a continuous function. We say that $(X_t)_{t\ge0}$ is a stationary solution to \eqref{fractionalSDE}
if
\begin{equation}\label{eq-ids}
(X_t - X_0 - \int_0^t b(X_s) ds)_{t \ge 0}  \overset{\cal L}{=} ( Z_t)_{t \ge 0}, 
\end{equation}
where the equality is the equality of all finite dimensional margins,
and if for every $n\in\EN$, for every $0\le t_1<t_2<\ldots<t_n$,
$$(X_{t+t_1},\ldots,X_{t+t_n})\overset{\cal L}{=}(X_{t_1},\ldots,X_{t_n})\quad\forall t\ge0,$$
where $\overset{\cal L}{=}$ denotes the equality in distribution.
\end{definition}

 \begin{Remarque}Since $ Z $ has in general no integral representation like the moving average
representation of the fractional Brownian motion 
$$ B_H(t)=\int_{-\infty}^{\infty} (t-s)_+^{H-\frac12} - (-s)_+ dW_s$$ 
we don't have any stationary noise process in the sense of Definition 2.6 in~\cite{hairer}. 
Actually, our definition is closer to the classical definition of invariant measure 
of Random Dynamical System (see~\cite{arnold,crauel}). 
 \end{Remarque}

\noindent When $ (Z_t)_{t\ge0} $ is a Markovian process, for instance a Brownian motion, it is classical 
to have $X_0$ independent of $Z,$ but in general we cannot have such independence 
as stated in Proposition~\ref{dependenceprop} (see Section \ref{sect-properties}). 
\begin{definition} \label{stat-sol-ini} Let $\nu$ denote a probability on $\ER^d$. We say that $\nu$ is an invariant distribution for \eqref{fractionalSDE0}
if there exists a stationary solution $(X_t)_{t\ge0}$ to \eqref{fractionalSDE0} such that 
$\nu={\cal L}(X_0)$.
 \end{definition}
 \begin{Remarque} The fact that $X_0$ and $(Z_t)_{t\ge0}$ may  be dependent involves that uniqueness of the invariant distribution does not imply uniqueness of stationary solutions to \eqref{eq-ids}.
 \end{Remarque}

\noindent Let $\gamma$ be a positive number. We will now discretize equation~\eqref{fractionalSDE} as follows:
\begin{equation*}
\begin{cases}
 Y_{(n+1)\gamma}-Y_{n\gamma}=\gamma b(Y_{n\gamma})+ \Delta_{n+1}\quad\forall n\ge0.\\
 Y_t=Y_{n\gamma}\quad\forall t\in[n\gamma,(n+1)\gamma).
\end{cases}\hspace{6.5cm}\mathbf{(E_\gamma)}
 \end{equation*}
We will say that $(Y_t)_{t\ge0}$ is a \textit{discretely} stationary solution  to~$\mathbf{(E_\gamma)}$ is solution of~$\mathbf{(E_\gamma)}$ satisfying:
$$ (Y_{t_1+k\gamma},\ldots,Y_{t_n+k\gamma})\overset{\cal L}{=}(Y_{t_1},\ldots,Y_{t_n})\quad\forall\, 0<t_1<\ldots<t_n,\forall n,k\in\EN.$$
\noindent We denote $(\bar{X}_{\Gn})$ the Euler scheme defined by:
 $\bx_0=x\in\ER^d$ and for every $n\ge0$
 \begin{equation}
   \label{eq:fractionalSDE-disc}
   \bx_{\Gnp}=\bx_\Gn+\gamma b(\bx_\Gn)+\Delta_{n+1}.
 \end{equation}
Then, we denote by $(\bx_t)_{t\ge0}$ the stepwise constant continuous-time Euler scheme defined by:
 $$\bx_t=\bx_{\Gn} \quad\forall t\in[\Gn,\Gnp).$$
The process $ (\bar{X}_t)_{t\ge0}$ is a solution to~$\mathbf{(E_\gamma)}$ such that $ \bar{X}_0=x.$
 For every $k\ge 0$, we define by $(\bar{X}_t^{(\gamma k)})_{t\ge 0}$ the
 $(\gamma k)$-shifted process: $\bar{X}_t^{(\gamma k)}=\bar{X}_{\gamma k+t}.$\\
 Then, a sequence of random
 probability measures $({\cal P}^{(n)}(\omega,d\alpha))_{n\ge 1}$ is defined  on the Skorokhod space
 $\mathbb{D}(\ER_+,\ER^d)$ by
 $${\cal P}^{(n,\gamma)}(\omega,d\alpha)=\frac{1}{n}\sum_{k=1}^n
 {\delta}_{{\bar{X}}^{({\gamma(k-1)})}(\omega)},
(d\alpha)$$ where 
 $\delta$ denotes the Dirac measure. 
For $t\ge 0$,
 the sequence $({\cal P}_t^{(n)}(\omega,dy))_{n\ge 1}$ of
 \guil marginal" empirical measures  at time $t$ on $\ER^d$ is defined by
 $${\cal P}^{(n,\gamma)}_t(\omega,dy)=\frac{1}{n}\sum_{k=1}^n
 {\delta}_{\bar{X}^{(\gamma(k-1))}_t(\omega)}(dy)=\frac{1}{n}\sum_{k=1}^n
 {\delta}_{\bar{X}_{\gamma(k-1)+t}(\omega)}(dy).$$

\noindent A weak limit of  a set ${\mathcal P} \subset \mathbb{D}(\ER_+,\ER^d)$ is a limit of any subsequence 
of $ {\mathcal P} $ in $\mathbb{D}(\ER_+,\ER^d).$ Let us now state the main results. 
 \begin{theorem}\label{principal} 1. Assume $\mathbf{(H_1)}$ and $\mathbf{(H_2)}$. Then, there exists $\gamma_0>0$ such that for every $\gamma\in(0,\gamma_0)$, $({\cal P}^{(n,\gamma)}(\omega,d\alpha))_{n\ge1}$ is $a.s.$ tight on $\mathbb{D}(\ER_+,\ER^d)$. Furthermore,
 every weak limit is a discretely stationary solution to~$\mathbf{(E_\gamma)}$.\\
 2.  Assume   $\mathbf{(H_1)}$ and $\mathbf{(H_3)}$ and set 
 $${\cal U}^{\infty,\gamma}(\omega):=\{\textnormal{weak limits of 
$({\cal P}^{(n,\gamma)}(\omega,d\alpha))$}\}.$$
 Then, there exists $\gamma_1\in(0,\gamma_0)$ such that $({\cal U}^{\infty,\gamma}(\omega))_{\gamma\le\gamma_1}$ is a.s. relatively compact for the uniform convergence topology on compact sets and any weak limit when $\gamma\rightarrow0$ of $({\cal U}^{\infty,\gamma}(\omega))_{\gamma\le\gamma_1}$ is a stationary solution to \eqref{fractionalSDE}. 
 \end{theorem} 
 The previous theorem states existence of  stationary solutions of~\eqref{fractionalSDE}, but one can 
wonder about uniqueness of the solutions. We will only consider the special case when $\mathbf{(H_{3,0})}$  is enforced which is 
called in the Markovian setting  asymptotic confluence (By asymptotic confluence, we mean that the distance in probability between two solutions starting from two different points $x$ and $y$ tends to 0 when $t\rightarrow+\infty$). 

\begin{prop}
\label{uniq} Assume $\mathbf{(H_1)}$ and  $\mathbf{(H_{3,0})}$. Then, there exists a unique stationary solution to \eqref{fractionalSDE}   and to equation $\mathbf{(E_\gamma)}$, when $\gamma$ is small enough. 
\end{prop}
\noindent The next corollary, whose proof is obvious is nevertheless useful.
\begin{Corollaire} 
Assume $\mathbf{(H_1)}$ and $\mathbf{(H_{3,0})}$. Denote by $\mu\in{\cal P}({\cal C}(\ER_+,\ER^d))$, the distribution of the unique stationary solution to \eqref{fractionalSDE0}. Then, 
  \begin{equation}
  d_{\mathbb{D}(\ER_+,\ER^d)}({\cal P}^{(\infty,\gamma)}(\omega,d\alpha),\mu)\xrn{\gamma\rightarrow0}0\quad a.s.  
  \end{equation}
  where $d_{\mathbb{D}(\ER_+,\ER^d)}$ denotes a distance on ${\cal P}(\mathbb{D}(\ER_+,\ER^d))$ (endowed with the weak topology), the set of probabilities on $\ER^d.$
 In particular,
  \begin{equation}
    \label{eq:conv}
 d_{\ER^d}({\cal P}^{(\infty,\gamma)}(\omega,d\alpha),\nu)\xrn{\gamma\rightarrow0}0\quad a.s.  
  \end{equation}
  where $\nu$ is the  unique  invariant distribution of~\eqref{fractionalSDE}  and, $d_{\ER^d}$ is a distance on ${\cal P}(\ER^d)$.
 \end{Corollaire}

 \noindent 
We will not study the rate of convergence relative to~\eqref{eq:conv} in this paper.\\
\begin{Remarque} We chose in this paper to work with the stepwise constant Euler scheme because this continuous-time scheme is in a sense the simplest to manage. The default is that the previous convergence result is stated for the Skorokhod topology. Replacing the stepwise constant Euler scheme by a continuous-time Euler scheme built by interpolations would lead to a convergence result for the topology of uniform convergence on compact sets.
\end{Remarque}

\noindent
Although  $Z$ is not supposed to have an explicit  integral representation with respect to  a Wiener process, and that   the setting of the Stochastic Dynamical 
System (in short SDS) of~\cite{hairer} seems  hard to use in our work,  let us start a brief comparison of our results 
with those of~\cite{hairer} if $(Z_t)_{t\ge0}$ is fractional Brownian motion.
First, our assumption  $\mathbf{(H_{3})}$ is a stability assumption a little weaker than  $\mathbf{(A_{1})}$ in~\cite{hairer}.
Likewise  $\mathbf{(H_{2})}$ $(i)$ and $b$ Lipschitz continuous are   similar to $\mathbf{(A_{2})}$ for $ N=1$ 
with Hairer's notation.  In~\cite{hairer} Stochastic Dynamical System (SDS Definition 2.7) and a Feller semigroup 
$ {\mathcal Q}_t $ ((2.4) in~\cite{hairer}) are defined on $  \ER^d \times { \cal C}(\ER_+,\ER^d).$ 
The first marginal of a stationary measure $\mu$ on $  \ER^d \times { \cal C}(\ER_+,\ER^d)$ defined in section 2.3 of~\cite{hairer} is what
we call an invariant measure in Definition~\ref{stat-sol-ini}. 
Moreover $ {\cal P}^{(n,\gamma)}$  for large $n$ and small $ \gamma$ are natural approximations of the stationary measures
of~\cite{hairer}.

\section{Tightness of $({\cal P}^{(n,\gamma)}(\omega,d\alpha))_{n\ge1}$}\label{tightness1}
The main result of this section is Proposition \ref{prop1} where we show the first part of Theorem \ref{principal}, $i.e.$ we obtain that  $({\cal P}^{(n,\gamma)}(\omega,d\alpha))_{n\ge1}$
is $a.s.$ tight for the Skorokhod topology on $\mathbb{D}(\ER_+,\ER^d)$ when $\gamma$ is sufficiently small. A fundamental step for this proposition is to obtain the  $a.s.$ tightness for the sequence of initial distributions $({\cal P}^{(n,\gamma)}_0(\omega,d\alpha))_{n\ge1}$. This property is established in the following lemma.
\begin{lemme}\label{lemme3} Assume $\mathbf{(H_1)}$ and $\mathbf{(H_2)}$. Then, there exists $\gamma_0>0$ such that for every $\gamma\le \gamma_0$,
\begin{equation}\label{eq20}
\sup_{n\ge1}\frac{1}{n}\sum_{k=1}^n|\bar{X}_{\gamma({k-1})}|^2<+\infty\quad a.s.
\end{equation}
\end{lemme}
\begin{proof}
We have :
\begin{align*}
|\bar{X}_{{(n+1)\gamma}}|^2&=|\bar{X}_{\Gn}|^2+2\gamma \psg\bar{X}_{\Gn}, b(\bar{X}_{\Gn})\psd+2\psg\bar{X}_{\Gn},\Delta_{n+1}\psd\\
&+\left(\gamma^2 |b(\bar{X}_{\Gn})|^2+2\gamma \psg b(\bar{X}_{\Gn}),\Delta_{n+1}\psd+|\Delta_{n+1}|^2\right).
\end{align*}
Let $\varepsilon>0$. Using assumption $\mathbf{(H_2)}(i)$ and the elementary inequality $|\psg u,v\psd|\le \frac{1}{2}(|\varepsilon u|^2+|v/\epsilon|^2)$ (for every $u,v\in\ER^d$), we have:
\begin{align*}
&|\psg\bar{X}_{\Gn}, \Delta_{n+1}\psd|\le\frac{1}{2}\left(\varepsilon|\bar{X}_{{n\gamma}}|^2+\frac{1}{\varepsilon}| \Delta_{n+1}|^2\right) \quad \textnormal{and}\\
&|\psg b(\bar{X}_{\Gn}),\Delta_{n+1}\psd|\le\frac{1}{2}\left(\varepsilon C(1+|\bar{X}_{{n\gamma}}|^2)+\frac{1}{\varepsilon}| \Delta_{n+1}|^2\right).
\end{align*}
It follows from assumption $\mathbf{(H_2)}(ii)$ that for every $\varepsilon>0$, 
\begin{equation*}
|\bar{X}_{{(n+1)\gamma}}|^2\le |\bar{X}_{\Gn}|^2+2\gamma(\beta-\alpha |\bar{X}_{\Gn}|^2)+p(\gamma,\varepsilon)(1+|\bar{X}_{{n\gamma}}|^2)+C(\varepsilon,\gamma)|\Delta_{n+1}|^2
\end{equation*}
where $C(\gamma,\varepsilon)$ is a positive constant depending on $\gamma$ and $\varepsilon$ and $p(\gamma,\varepsilon)\le C(\varepsilon+\gamma \varepsilon+\gamma^2)$. Then, set $\varepsilon=\gamma^2$ (for instance). For $\gamma$ sufficiently small,
$p(\gamma,\varepsilon)\le \alpha\gamma/2$. Hence, we obtain that there exist $\tilde{\beta}\in\ER$ and $\tilde{\alpha}>0$ such that
$\forall n\ge0$
 \begin{align}
|\bar{X}_{{(n+1)\gamma}}|^2&\le |\bar{X}_{\Gn}|^2+\gamma(\tilde{\beta}-\tilde{\alpha}|\bar{X}_{\Gn}|^2)+C|\Delta_{n+1}|^2\quad \nonumber\\
&\le (1-\gamma\tilde{\alpha})|\bar{X}_{\Gn}|^2+C(\gamma+|\Delta_{n+1}|^2).\label{ineq1}
\end{align}
Finally, by induction, one obtains for every $n\ge1$:
$$|\bar{X}_{\Gn}|^2\le (1-\gamma\tilde{\alpha})^n|x|^2+C\sum_{k=1}^{n}(1-\gamma\tilde{\alpha})^{n-k}(\gamma+|\Delta_{k}|^2).$$
Hence, in order to prove \eqref{eq20}, it is enough to show that for $\gamma$ sufficiently small,
\begin{equation}\label{eq21}
\sup_{n\ge1}\frac{1}{n}\sum_{k=1}^{n-1}\sum_{l=1}^{k}(1-\tilde{\alpha}\gamma)^{k-l}|\Delta_l|^2<+\infty\quad a.s.
\end{equation}
But  
$$\sum_{k=1}^n\sum_{l=1}^{k}(1-\tilde{\alpha}\gamma)^{k-l}|\Delta_l|^2=\sum_{k=1}^n|\Delta_k|^2\sum_{u=0}^{n-k}(1-\tilde{\alpha}\gamma)^u\le C\sum_{k=1}^n|\Delta_k|^2\le C\sum_{i=1}^{d}\sum_{k=1}^n(\Delta_k^i)^2,$$
and it follows that it is in fact enough to show that 
\begin{equation}\label{eq22}
\sup_{n\ge1}\frac{1}{n}\sum_{k=1}^n(\Delta_k^i)^2<+\infty\quad a.s. \quad\forall i\in\{1,\ldots,{d}\}.
\end{equation}
\noindent Now,  by \eqref{ergodicite-increments},
$$\frac{1}{n}\sum_{k=1}^n(\Delta_k^i)^2\xrn{n\nrn}\ES[(\Delta_1^i)^2],$$
and \eqref{eq22} is satisfied. This completes the proof.
\end{proof}
\begin{prop}\label{prop1} Assume assumption $\mathbf{(H_1)}$ and $\mathbf{(H_2)}$. Then, there exists $\gamma_0>0$ such that for every $\gamma\le \gamma_0$, $({\cal P}^{(n,\gamma)}(\omega,d\alpha))_{n\ge1}$ is $a.s.$ tight on $\mathbb{D}(\ER_+,\ER^d)$.
\end{prop}
\begin{proof}

We have to prove the two following points (see $e.g.$ \cite{billingsley}, Theorem 15.2):
\begin{itemize}
\item{1.} $\forall T>0$, $(\mu_T^{(n)}(\omega,dy))$ defined by
$$\mu_T^{(n)}(\omega,dy)=\sum_{k=1}^n\delta_{\{\sup_{t\in[0,T]}|\bar{X}^{(k-1)}_t|\}}(dy),$$
is an $a.s.$ tight sequence.
\item{2.} For every $\eta>0$, 
$$\limsup_{\delta\rightarrow0}\limsup_{n\rightarrow+\infty}\frac{1}{n}\sum_{k=1}^n\delta_{\{\omega'_T(\bar{X}^{(k-1)},\delta)\ge \eta\}}=0 \quad a.s.$$
with
$$w'_T(x,\delta)=\underset{\{t_i\}}{\inf}\{\max_{i\le r}\sup_{s,t\in[t_i,t_{i+1})}|x_t-x_s|\}$$
where  the infimum extends over finite sets $\{t_i\}$ satisfying:
$$0=t_0<t_1<\ldots<t_r=T\quad \textnormal{and}\quad\inf_{i\le r}(t_i-t_{i-1})\ge\delta.$$
\end{itemize}
In fact, since the process has only jumps at times $n\gamma$ with $n\in\EN$, $\omega'_T(\bar{X}^{(k)},\delta)=0$ when $\delta<\gamma$. It follows that the second point is obvious. Then, let us prove the first point.
By induction, one gets from \eqref{ineq1} that, for every $k\ge n$,
$$|\bar{X}_{k\gamma}|^2\le |\bar{X}_{n\gamma}|^2 (1-\gamma\tilde{\alpha})^{k-n}+C\sum_{l=n+1}^{k}(1-\gamma\tilde{\alpha})^{k-l}(\gamma+|\Delta_{l}|^2).$$
This implies that 
$$\sup_{t\in[0,T]}|\bar{X}^{(k-1)}_t|^2=\sup_{k\in\{n,\ldots,n+[T/\gamma]\}}|\bar{X}_{k\gamma}|^2\le|\bar{X}_{n\gamma}|^2+C(1+\sum_{l=n+1}^{n+[T/\gamma]}|\Delta_l|^2).$$
Thus, if $V(x)=|x|^2,$ one can deduce:
\begin{align*}
\mu^{(n)}(\omega,V)&\le \frac{1}{n}\sum_{k=1}^{n} V(\bar{X}_{(k-1)\gamma})+C\Big(1+\frac{1}{n}\sum_{k=1}^{n}\sum_{l=k+1}^{k+[T/\gamma]}|\Delta_l|^2\Big)\\ 
&\le\sup_{n\ge1}{\cal P}_0^{(n,\gamma)}(\omega,V)+C\Big(1+\frac{1}{n}\left[\frac{T}{\gamma}\right]\sup_{n\ge1}\frac{1}{n}\sum_{k=1}^{n+[T/\gamma]}|\Delta_k|^2\Big)<+\infty\quad a.s.
\end{align*}
thanks to Lemma \ref{lemme3} and \eqref{eq22}. Therefore, $\sup_{n\ge1}\mu^{(n)}_T(\omega,V)<+\infty$ $a.s$ which implies that 
$(\mu^{(n)}_T(\omega,dy))$ is $a.s.$ tight on $\ER^d$ (see $e.g.$. \cite{duflo}, Proposition 2.1.6).
\end{proof}
\section{Identification of the weak limits of $({\cal P}^{(n,\gamma)}(\omega,d\alpha))_{n\ge1}$} \label{identification1}
In the following proposition, we show that every weak limit of $({\cal P}^{(n,\gamma)}(\omega,d\alpha))_{n\ge1}$ is $a.s$
a stationary Euler scheme with step $\gamma$ of SDE \eqref{fractionalSDE}.
\begin{prop}\label{prop2} Assume $\mathbf{(H_1)}$ and let ${\cal P}^{(\infty,\gamma)}(\omega,d\alpha)$ denote a weak limit of $({\cal P}^{(n,\gamma)}(\omega,d\alpha))_{n\ge1}$. Then, a.s., 
${\cal P}^{(\infty,\gamma)}(\omega,d\alpha)$ is the distribution of a càdlàg process denoted by $Y^{(\infty,\gamma)}$ such that, $a.s.$ in $\omega$,
\begin{itemize}
\item[(a)] $(Y^{(\infty,\gamma)}_{l\gamma+t})_{t\ge0}\overset{\mathbb{D}(\ER_+,\ER^d)}{=}(Y^{(\infty,\gamma)}_t)_{t\ge0}$ for every $l\in\EN$ where $\overset{\mathbb{D}(\ER_+,\ER^d)}{=}$ denotes the equality in distribution on $\mathbb{D}(\ER_+,\ER^d)$. 
\item[(b)] $N^{(\infty,\gamma)}$
defined by
$$N^{(\infty,\gamma)}_t=Y^{(\infty,\gamma)}_t-Y^{(\infty,\gamma)}_0-\int_0^{\underline{t}_\gamma}b(Y^{(\infty,\gamma)}_s)ds$$
is equal in law to $ \bar{Z}^{\gamma}$ with $\underline{t}_\gamma=\gamma[t/\gamma].$
\end{itemize}
\end{prop}
\begin{Remarque}\label{DSSF}  It follows from the previous proposition that $(Y^{(\infty,\gamma)}_{t})_{t\ge0}$ is a discretely stationary solution to ~$\mathbf{(E_\gamma)}$.
\end{Remarque} 
\begin{proof}
\textit{(a)} Let ${\cal T}$ denote a countable dense subset of $\ER_+$ and ${\cal S}_r^K$, a countable dense subset of the space of continuous functions $f:\ER^r\rightarrow\ER$ with compact support.
It suffices to prove that $a.s.$, $\forall r\ge0$, for every $f\in{\cal S}_r^K$, for every 
$t_1,\ldots,t_r\in {\cal T}$,$\forall l\in\EN$, 
$$\int f(\alpha_{t_1},\ldots,\alpha_{t_r}){\cal P}^{(\infty,\gamma)}(\omega,d\alpha)=\int f(\alpha_{t_1+l\gamma},\ldots,\alpha_{t_r+l\gamma}){\cal P}^{(\infty,\gamma)}(\omega,d\alpha).$$
Since ${\cal T}$ and ${\cal S}_r^K$ are countable, we  only have to prove that  
 $\forall r\ge0$, for every $f\in{\cal S}_r^K$, for every 
$t_1,\ldots,t_r\in {\cal T}$, $\forall l\in\EN$,
\begin{equation}\label{statfinite}
\int f(\alpha_{t_1},\ldots,\alpha_{t_r}){\cal P}^{(\infty,\gamma)}(\omega,d\alpha)=\int f(\alpha_{t_1+l\gamma},\ldots,\alpha_{t_r+l\gamma}){\cal P}^{(\infty,\gamma)}(\omega,d\alpha)\quad a.s.
\end{equation}
Let now $f\in{\cal S}_r^K$, $l\in\EN$ and $t_1,\ldots,t_r\in{\cal T}$. On the one hand,
\begin{align*}
&\frac{1}{n}\sum_{k=1}^{n}\left(f(\bar{X}_{t_1}^{(k-1)},\ldots,\bar{X}_{t_r}^{(k-1)})-f(\bar{X}_{t_1+l\gamma}^{(k-1)},\ldots,\bar{X}_{t_r+l\gamma}^{(k-1)})\right)\\
&=\frac{1}{n}\sum_{k=1}^{n}f(\bar{X}_{(k-1)\gamma+t_1},\ldots,\bar{X}_{(k-1)\gamma+t_r})-
\frac{1}{n}\sum_{k=1}^{n}f(\bar{X}_{(k-1+l)\gamma+t_1},\ldots,\bar{X}_{(k-1+l)\gamma+t_r})\\
&=\frac{1}{n}\left(\sum_{k=1}^{l-1}f(\bar{X}_{(k-1)\gamma+t_1},\ldots,\bar{X}_{(k-1)\gamma+t_r})-\sum_{k=n+1}^{n+l}f(\bar{X}_{(k-1)\gamma+t_1},\ldots,\bar{X}_{(k-1)\gamma+t_r})\right),
\end{align*}
and this last term converges to 0 when $n\rightarrow+\infty$ $a.s.$ since $f$ is bounded. 
On the other hand, since ${\cal P}^{(\infty,\gamma)}(\omega,d\alpha)$ denotes a weak limit of $({\cal P}^{(n,\gamma)}(\omega,d\alpha))_{n\ge1}$, there exists  a subsequence
$(n_k(\omega))_{k\ge1}$ such that $({\cal P}^{(n_k(\omega),\gamma)}(\omega,d\alpha))_{k\ge1}$ converges weakly to ${\cal P}^{(\infty,\gamma)}(\omega,d\alpha)$ (for the Skorokhod topology). This convergence implies in particular the finite-dimensional convergence. Therefore,
\begin{align*}
&\frac{1}{n}\sum_{k=1}^{n}f(\bar{X}_{t_1}^{(k-1)},\ldots,\bar{X}_{t_r}^{(k-1)})\xrn{n\nrn}\int f(\alpha_{t_1},\ldots,\alpha_{t_r}){\cal P}^{(\infty,\gamma)}(\omega,d\alpha)\quad a.s.\; \\
&and,\quad
\frac{1}{n}\sum_{k=1}^{n}f(\bar{X}_{t_1+l\gamma}^{(k-1)},\ldots,\bar{X}_{t_r+l\gamma}^{(k-1)})\xrn{n\nrn}\int  f(\alpha_{t_1+l\gamma},\ldots,\alpha_{t_r+l\gamma}){\cal P}^{(\infty,\gamma)}(\omega,d\alpha)\quad a.s.
\end{align*}
Therefore, \eqref{statfinite} follows.\\

\noindent \textit{(b)} Let $\Phi_\gamma:\mathbb{D}(\ER_+,\ER^d)\rightarrow\mathbb{D}(\ER_+,\ER^d)$ be defined by:
\begin{equation}\label{phigamma}
(\Phi_\gamma(\alpha))_t=\alpha_t-\alpha_0-\int_0^{\underline{t}_\gamma}b(\alpha_s)ds.
\end{equation}
Then, $N^{(\infty,\gamma)}=\Phi_\gamma(Y^{(\infty,\gamma)})$. Let   $F:\mathbb{D}(\ER_+,\ER^d)\rightarrow\ER$ be a bounded continuous functional:
\begin{equation}\label{funcfrac}
\ES[F(N^{(\infty,\gamma)})]=\int F(\Phi_\gamma(\alpha)){\cal P}^{(\infty,\gamma)}(\omega,d\alpha)=\lim_{k\rightarrow+\infty}\int F(\Phi_\gamma(\alpha)){\cal P}^{(n_k(\omega),\gamma)}(\omega,d\alpha).
\end{equation}
For every $t\ge0$,
$$ \Phi_\gamma(\bar{X}^{(k)})_t=(\bar{Z}^\gamma_{{\gamma_k+t}}-\bar{Z}^\gamma_{k\gamma})=\sum_{l=k+1}^{k+[t/\gamma]}\Delta_l,$$
with the convention $\displaystyle{\underset{\emptyset}{\sum}=0}$. Thus, we derive from \eqref{funcfrac} that
$$\ES[F(N^{(\infty,\gamma)})]=\lim_{k\rightarrow+\infty}\frac{1}{n_k}\sum_{m=1}^{n_k} F\circ G((\Delta_l)_{l\ge m})$$
where $G:(\ER^d)^\EN\rightarrow\mathbb{D}(\ER_+,\ER^d)$ is defined by
$$G((u_n)_{n\ge1})_t=\sum_{l=1}^{[\frac{t}{\gamma}]} u_l\quad\forall t\ge0.$$
Now, $(\Delta_n)_{n\ge1}$ is an ergodic sequence (see \eqref{ergodicite-increments}). 
As a consequence, $a.s.$,
$$\frac{1}{n_k}\sum_{i=1}^{n_k} F\circ G((\Delta_l)_{l\ge i})\xrightarrow{k\rightarrow+\infty}\ES[F\circ G((\Delta_l)_{l\ge1})]=
\ES[F( \bar{Z}^{\gamma})].$$
The result follows.
\end{proof}
\section{Convergence of $({\cal P}^{(\infty,\gamma)}(\omega,d\alpha))$ when $\gamma\rightarrow0$}\label{identification2}
The aim of this section is to show that, $a.s.$, $({\cal P}^{(\infty,\gamma)}(\omega,d\alpha))_\gamma$ is $a.s.$ tight for the weak topology induced by the topology of uniform convergence on $\mathbb{D}(\ER_+,\ER^d)$ and that its weak limits when $\gamma\rightarrow0$ are stationary solutions to \eqref{fractionalSDE}.  The main difficulty for this second part of the proof of Theorem \ref{mainresult} is to show that  $({\cal P}^{(\infty,\gamma)}(\omega,d\alpha))_\gamma$ is $a.s.$ tight on $\ER^d$. For this step, we focus in Lemma \ref{lemme2} on the particular case $b(x)=-x$ (when $(X_t)_{t\ge0}$ is an Ornstein-Uhlenbeck process) where
some explicit computations lead to a control of $({\cal P}^{(\infty,\gamma)}(\omega,d\alpha))_\gamma$. Then, in Lemma \ref{lemme4}, we show that this control can be extended to SDE's whose drift term satisfies  $\mathbf{(H_3)}$.
Finally, we establish the main result of this section in Proposition \ref{propid2}.\\

\noindent Let $\gamma>0$. We denote by $(\Sigma_{n\gamma})$ the Euler scheme in the particular case $b(x)=-x$. We have $\Sigma_0=x$ and:
$$\Sigma_{(n+1)\gamma}=(1-\gamma)\Sigma_{n\gamma}+\Delta_{n+1}\quad\forall n\ge0.$$
\begin{lemme}\label{lemme2}
Assume $\mathbf{(H_1)}$ and let $\gamma\in(0,1)$. Then, $(\ES[|\Sigma_{{n\gamma}}|^2])_{n\ge0}$ is a convergent sequence. Denote by $v(\gamma)$ its limit. For every $\gamma_0\in(0,1)$, 
$$\sup_{\gamma\in(0,\gamma_0]} v(\gamma)<+\infty.$$
\end{lemme}
\begin{proof}  
First, by induction,
$$\Sigma_{{n\gamma}}=(1-\gamma)^n x+\sum_{k=0}^{n-1} (1-\gamma)^{k}\Delta_{n-k}.$$

It follows that
$$\ES[|\Sigma_{{n\gamma}}|^2]=(1-\gamma)^{2n}|x|^2+\sum_{i=1}^{d} \ES\Big[\Big(\sum_{k=0}^{n-1} (1-\gamma)^{k}\Delta_{n-k}^i\Big)^2\Big].$$
For every $i\in\{1,\ldots,{d}\}$,
$$\ES\Big[\Big(\sum_{k=0}^{n-1} (1-\gamma)^{k}\Delta_{n-k}^i\Big)^2\Big]=\sum_{k=0}^{n-1}\sum_{l=0}^{n-1}(1-\gamma)^{k+l}\phi_\gamma^i(l-k),$$
where $\phi_\gamma^i$ is defined by \eqref{covariance}. Setting $u=k+l$ and $v=l-k$, we deduce that
\begin{equation}\label{587}
\ES\Big[\Big(\sum_{k=0}^{n-1} (1-\gamma)^{k}\Delta_{n-k}^i\Big)^2\Big]=\sum_{u=0}^{2n-2} (1-\gamma)^u\sum_{v=(u-(2n-2))\vee (-u)}^{(2n-2-u)\wedge u}\phi_\gamma^i(v),
\end{equation}
with $x\wedge y=\min(x,y)$ and $x\vee y=\max(x,y)$. Then, with the definition of $\phi$, one can check that
\begin{equation*}
\sum_{v=(u-(2n-2))\vee (-u)}^{(2n-2-u)\wedge u}\phi_\gamma^i(v)=\begin{cases}{c}_i(\gamma)&\textnormal{if $u=0$ or $u=2n-2$,}\\
f_i^\gamma((2n-2-u)\wedge u)&\textnormal{otherwise,}
\end{cases}
\end{equation*}
with $f_i^\gamma(x)={c}_i(\gamma(x+1))-{c}_i(\gamma x).$ It follows from \eqref{587} that,
$$\ES\Big[\Big(\sum_{k=0}^{n-1} (1-\gamma)^{k}\Delta_{n-k}^i\Big)^2\Big]={c}_i(\gamma)+\sum_{u=1}^{n-1}(1-\gamma)^{u} f^\gamma_i(u)+R_n(\gamma),$$
with
\begin{align*}
R_n(\gamma)&=\sum_{u=n}^{2n-1}(1-\gamma)^{u}f_i^\gamma(2n-2-u)+(1-\gamma)^{2n-2}{c}_i(\gamma),\\
&=\sum_{u=-1}^{n-2}(1-\gamma)^{2n-2-u}f_i^\gamma(u)+(1-\gamma)^{2n-2}{c}_i(\gamma).
\end{align*}
Since ${c}_i$ is locally bounded and ${c}_i''$ is bounded on $[1,+\infty[$, ${c}_i$ is a subquadratic function, $i.e.$
$$|{c}_i(u)|\le C(1+|u|^2)\quad\forall u\ge0.$$
It follows that $f_i^\gamma$ is also a subquadratic function. Then, using that for every $u\in\{-1,\ldots,n-2\}$, $(1-\gamma)^{2n-2-u}\le(1-\gamma)^n$, we obtain that for every $\gamma\in(0,1)$,
$R_n(\gamma)\longrightarrow0$ as $n\rightarrow+\infty$.\\
Using again that $f_i^\gamma$ is a subquadratic function, we deduce that for every $\gamma\in(0,1)$, for every $i\in\{1,\ldots,d\}$,
$$\ES\Big[\Big(\sum_{k=0}^{n-1} (1-\gamma)^{k}\Delta_{n-k}^i\Big)^2\Big]\xrn{n\nrn}w_i(\gamma):={c}_i(\gamma)+\sum_{u=1}^{+\infty}(1-\gamma)^{u} f_\gamma^i(u)$$
and that $w_i(\gamma)$ is finite.
By a second order Taylor development, we have for every $u\ge1$: 
\begin{equation*}
f_i^\gamma(u)=\gamma {c}_i'(\gamma u)+\gamma^2 r(\gamma,u)\quad
\textnormal{with}\quad r(\gamma,u)=c''_i(\gamma(u+\theta_u)),\quad\theta_u\in[0,1].
\end{equation*}
Hence, using assumption $\mathbf{(H_1)}$, it follows that
\begin{align*}
&w_i(\gamma)={c}_i(\gamma)+\sum_{u=1}^{+\infty} \gamma(1-\gamma)^u\left[{c}'_i(\gamma u)+\gamma r(\gamma,u)\right]\quad\textnormal{with} \quad |r(\gamma,u)|\le C g_{i,1}(\gamma u),\\
&\textnormal{and,}\quad g_{i,1}(t)= t^{-a_i}1_{\{t\in(0,1)\}}+t^{-b_i}1_{\{t\ge1\}}.
\end{align*}
Let us now control the behavior of $w_i(\gamma)$ when $\gamma\rightarrow0$. First, for every $\gamma\in(0,1)$, for every $u\ge1$,
$(1-\gamma)^u\le \exp(-\gamma u)$. Then, since $t\mapsto \exp(-t)$, $t\mapsto g_{i,1}(t)$ are non-increasing on $\ER_+^*$, one deduces that for every $u\ge2$,
$$ \gamma(1-\gamma)^u g_{i,1}(\gamma u)\le
\int_{\gamma(u-1)}^{\gamma u} \exp(-t) g_{i,1}(t)dt.$$
Then, 
$$|\sum_{u=1}^{+\infty}\gamma^2(1-\gamma)^u r(\gamma,u)|\le {c}_i(\gamma)(1-\gamma)+C\gamma\int_\gamma^{+\infty}\exp(-t) g_{i,1}(t)dt.$$
Using that $a_i<2$, we easily check that the right-hand side is bounded and  tends to 0 when $\gamma\rightarrow0$. 
We now focus on the first term of $w_i(\gamma)$. First, by assumption $\mathbf{(H_1)}$, for every $t>0$,
\begin{equation}\label{equ'}
|{c}'_i(t)|\le C(1+g_{i,2}(t))\quad\textnormal{where}\quad g_{i,2}(t)=t^{1-a_i-\delta_1}1_{\{t\in(0,1)\}}+t^{1-b_i+\delta_2}1_{\{t\ge1\}},
\end{equation}
with $\delta_1\in(0,1)$ (resp. $\delta_2\in(0,1)$)  if $a_i= 1$ (resp. $b_i=1$) and $\delta_1=0$ (resp. $\delta_2=0$) otherwise. 
Second, using that $(1-\gamma)^u\le C (1-\gamma)^{-1}\exp(-t)$ for every $t\in[\gamma u,\gamma (u+1)]$, one deduces that
\begin{equation}\label{eq''}
\gamma(1-\gamma)^u (\gamma u)^\rho\le C\begin{cases}\int_{\gamma(u-1)}^{\gamma u} \exp(-t) t^{\rho}dt&\textnormal{if $\rho<0$}\\
\frac{1}{1-\gamma}\int_{\gamma u}^{\gamma (u+1)} \exp(-t) t^{\rho}dt&\textnormal{if $\rho\ge 0$}
\end{cases}
\end{equation}
It follows from  \eqref{equ'} and \eqref{eq''} that 
\begin{equation*}
\limsup_{\gamma\rightarrow0}\left|\sum_{u=1}^{+\infty} \gamma(1-\gamma)^u{c}'_i(\gamma u)\right|\le C\int_0^{+\infty}\exp(-t) \left(1+g_{i,2}(t)\right)dt.
\end{equation*}
The right-hand member is finite. This completes the proof.
\end{proof}
\begin{lemme} \label{lemme4} Assume $\mathbf{(H_1)}$ and $\mathbf{(H_3)}$ and denote by ${\cal P}^{(\infty,\gamma)}(\omega,d\alpha)$ a weak limit of $({\cal P}^{(n,\gamma)}(\omega,d\alpha))$. Then:\\ 
\noindent (i) With the notations of Proposition \ref{prop2}, there exists $\gamma_0>0$ such that, 
\begin{equation}\label{contloiinitiale}
\sup_{0<\gamma\le\gamma_0}\ES_\omega[|Y^{(\infty,\gamma)}_0|^2]<+\infty\quad a.s.
\end{equation} 
\noindent (ii) Assume $\mathbf{(H_1)}$ and $\mathbf{(H_{3,0})}$. Then, uniqueness holds for the distribution of stationary solutions to \eqref{eq-ids}. Similarly, there exists $\gamma_0>0$ such that for every $\gamma\le \gamma_0$, uniqueness holds for the distribution of discretely stationary solutions to $\mathbf{(E_\gamma)}$.
\end{lemme}
\begin{proof} $(i)$\\
\noindent\textbf{Step 1}: 
Let $(\bar{X}_{n\gamma})$ and $(\Sigma_{n\gamma})$ be defined by: 
\begin{align}
&\label{schemex}\bx_0=x,\quad 
\bx_{\Gnp}=\bx_\Gn+\gamma b(\bx_\Gn)+\Delta_{n+1} \quad \textnormal{and},\\
&\bz_0=x,\quad 
\bz_{\Gnp}=\bz_\Gn-\gamma \bz_\Gn+\Delta_{n+1}\nonumber.
\end{align}
with $\Delta_n=Z_{n\gamma}-Z_{(n-1)\gamma}.$
Then,
\begin{align*}
|\bar{X}_{(n+1)\gamma}-&\Sigma_{(n+1)\gamma}|^2=|\bar{X}_{{n\gamma}}-\Sigma_{{n\gamma}}|^2+2\gamma\psg b(\bar{X}_{n\gamma})+\Sigma_{n\gamma},\bar{X}_{n\gamma}-\Sigma_{n\gamma}\psd
+\gamma^2|b(\bar{X}_{n\gamma})+\Sigma_{n\gamma}|^2\\
&\quad\le|\bar{X}_{{n\gamma}}-\Sigma_{{n\gamma}}|^2+2\gamma\psg b(\bar{X}_{n\gamma})-b(\Sigma_{n\gamma}),\bar{X}_{n\gamma}-\Sigma_{n\gamma}\psd
+2\gamma^2|b(\bar{X}_{n\gamma})-b(\Sigma_{n\gamma})|^2\\
&\quad+2\gamma\psg b(\Sigma_{n\gamma})+\Sigma_{n\gamma},\bar{X}_{n\gamma}-\Sigma_{n\gamma}\psd+2\gamma^2|b(\Sigma_{n\gamma})+\Sigma_{n\gamma}|^2.
\end{align*}
On the one hand, using that $b$ is  Lipschitz continuous and assumption $\mathbf{(H_3)}$, one obtains:
\begin{equation}\label{firstpa}
\gamma\psg b(\bar{X}_{n\gamma})-b(\Sigma_{n\gamma}),\bar{X}_{n\gamma}-\Sigma_{n\gamma}\psd
+2\gamma^2|b(\bar{X}_{n\gamma})-b(\Sigma_{n\gamma})|^2\le \gamma\left(\beta+|\bar{X}_{n\gamma}-\Sigma_{n\gamma}|^2(-\alpha+C\gamma)\right).
\end{equation}
On the other hand, using that $b$ is a sublinear function and  the elementary inequality $\psg u,v\psd\le1/2(\varepsilon^{-1}|u|^2+\varepsilon|v|^2)$ (with $u=b(\Sigma_{n\gamma})+\Sigma_{n\gamma}$, $v=\bar{X}_{n\gamma}-\Sigma_{n\gamma}$ and $\varepsilon=\alpha/2$), one also has:
\begin{equation}\label{secondpa}
\gamma\psg b(\Sigma_{n\gamma})+\Sigma_{n\gamma},\bar{X}_{n\gamma}-\Sigma_{n\gamma}\psd+2\gamma^2|b(\Sigma_{n\gamma})+\Sigma_{n\gamma}|^2\le
\gamma\frac{\alpha}{2}|\bar{X}_{n\gamma}-\Sigma_{n\gamma}|^2+C\gamma(1+|\Sigma_{n\gamma}|^2).
\end{equation}
Therefore, the combination of \eqref{firstpa} and \eqref{secondpa} yields for sufficiently small $\gamma$:
$$|\bar{X}_{(n+1)\gamma}-\Sigma_{(n+1)\gamma}|^2\le (1-\tilde{\alpha}\gamma)|\bar{X}_{n\gamma}-\Sigma_{n\gamma}|^2+C\gamma(1+|\Sigma_{n\gamma}|^2)$$
where $\tilde{\alpha}$ is a positive number.
Then, it follows from Lemma \ref{lemme2},
$$\ES[|\bar{X}_{(n+1)\gamma}-\Sigma_{(n+1)\gamma}|^2]\le (1-\tilde{\alpha}\gamma)\ES[|\bar{X}_{n\gamma}-\Sigma_{n\gamma}|^2]+\tilde{\beta}\gamma$$
where $\tilde{\beta}$ does not depend on $\gamma$. By induction, we obtain:
$$\sup_{n\ge1}\ES[|\bar{X}_{n\gamma}-\Sigma_{n\gamma}|^2]\le \tilde{\beta}\gamma\sum_{k=0}^{+\infty}(1-\tilde{\alpha}\gamma)^{k}=\frac{\tilde{\beta}}{\tilde{\alpha}}<+\infty.$$
Finally, since
$$\ES[|\bar{X}_{n\gamma}|^2]\le 2\left(\ES[|\bar{X}_{n\gamma}-\Sigma_{n\gamma}|^2]+\ES[|\Sigma_{n\gamma}|^2]\right),
$$
it follows from Lemma \ref{lemme2} that there exists $\gamma_0>0$ such that
\begin{equation}\label{controleX}
\sup_{0<\gamma\le\gamma_0}\sup_{n\ge1}\ES[|\bar{X}_{n\gamma}|^2]<+\infty.
\end{equation}
\textbf{Step 2}: First, since $\sup_{n\ge1}\frac{1}{n}\sum_{k=1}^n|\bar{X}_{(k-1)\gamma}|^2=\sup_{n\ge1}{\cal P}_0^{(n,\gamma)}(\omega,|x|^2)<+\infty$ $a.s.$ (by Lemma \ref{lemme3}), the fact that ${\cal P}^{(\infty,\gamma)}(\omega,d\alpha)$ is $a.s.$ a weak limit of ${\cal P}^{(n,\gamma)}(\omega,d\alpha)$ implies that
$$\ES_\omega[|Y^{(\infty,\gamma)}_0|^2]<+\infty\quad a.s.$$
By Proposition \ref{prop2}(b),  there exists $a.s.$ a Gaussian process $Z^\omega$ with the same distribution as the driving process of the SDE such that
\begin{equation}\label{trop}
Y^{(\infty,\gamma)}_{(n+1)\gamma}=Y^{(\infty,\gamma)}_{n\gamma}+\gamma b(Y^{(\infty,\gamma)}_{n\gamma})+ \Delta_{n+1}.
\end{equation}
with $\Delta_n:=Z^{\omega}_{n\gamma}-Z^{\omega}_{(n-1)\gamma}$. Moreover, by Remark \ref{DSSF}, $\ES_\omega[|Y^{(\infty,\gamma)}_{n\gamma}|^2]$ does not depend on $n$ since the sequence $(Y^{(\infty,\gamma)}_{n\gamma})$ is stationary.
Let now $(\bar{X}_{n\gamma}^x)$ be constructed as in \eqref{schemex}  with sequence $(\Delta_n)$ of \eqref{trop}.
By \eqref{controleX}, the lemma will be true if we are able to show that for sufficiently small $\gamma$,
\begin{equation}\label{asympconf}
\limsup_{n\rightarrow+\infty}\ES_\omega[|\bar{X}_{n\gamma}^x-Y^{(\infty,\gamma)}_{n\gamma}|^2]<C
\end{equation}
where $C$ does not depend on $\gamma$.
The process of the proof of \eqref{asympconf} is quite similar to  Step 1. First, using assumption $\mathbf{(H_3)}$, one checks that:
$$|\bar{X}_{(n+1)\gamma}-Y^{(\infty,\gamma)}_{(n+1)\gamma}|^2\le |\bar{X}_{n\gamma}-Y^{(\infty,\gamma)}_{n\gamma}|^2(1-\alpha\gamma+C\gamma^2)+\beta\gamma\quad a.s.$$
For sufficiently small $\gamma$, $\alpha\gamma-C\gamma^2\ge\gamma\alpha/2$. Setting $\tilde{\alpha}=\alpha/2$, one derives from an induction that:
\begin{equation*}
\ES_\omega[|\bar{X}_{(n+1)\gamma}-Y^{(\infty,\gamma)}_{(n+1)\gamma}|^2]\le (1-\hat{\alpha}\gamma)^n\ES_\omega[|\bar{X}_{0}^x-{Y}^{(\infty,\gamma)}_{0}|^2]+\beta\gamma\sum_{k=0}^{n-1}(1-\hat{\alpha}\gamma)^k\rightarrow\frac{\beta}{\hat{\alpha}}.
\end{equation*}
This concludes the proof of $(i)$.\\

\noindent (ii) First, we prove uniqueness for the distribution of a stationary solution to SDE \eqref{fractionalSDE}: let $(Y_{t,1})_{t\ge0}$ and $({Y}_{t,2})_{t\ge0}$ be some stationary solutions to  \eqref{fractionalSDE}  
driven respectively by  $Z^1$ and $Z^2.$ 
We want to show that for every $T>0$, for every bounded Lipschitz\footnote{for the standard distance $\delta$  defined for every $\alpha,\beta\in{\cal C}([0,T],\ER^d)$ by $\delta(\alpha,\beta)=\sup_{t\in[0,T]}|\alpha_t-\beta_t|.$} continuous functional $F:{\cal C}([0,T],\ER^d)\rightarrow\ER$,
\begin{equation}\label{egalloi}
\ES[F(Y_{t,1},0\le t\le T)]=\ES[F({Y}_{t,2}),0\le t\le T)].
\end{equation}
Let  $({X}_{t,1}^x)$ and $({X}_{t,2}^x)$ be some solutions to \eqref{fractionalSDE} starting from $x$ and built with the previous driving processes $Z^1$ and $Z^2$ respectively. First, since $b$ is Lipschitz continuous, a classical argument shows that weak uniqueness holds for solutions to \eqref{fractionalSDE} starting from any deterministic $x\in\ER^d$. As a consequence, $X_{.,1}^x$ and $X_{.,2}^x$ have the same distribution on ${\cal C}(\ER_+,\ER^d)$.  Thus, using that $Y_{.,1}$ and $Y_{.,2}$ are stationary, we obtain that  for every $s\ge0$:
 \begin{align*}
 \ES[F(Y_{.,1})]-\ES[F(Y_{.,2})]&=\ES[F(Y_{t+s,1},0\le t\le T)]-\ES[F({X}_{t+s,1}^x,0\le t\le T)]\\
 &+\ES[F({X}_{t+s,2}^x,0\le t\le T)]-\ES[F(Y_{t+s,2},0\le t\le T)].
 \end{align*}
 Since $F$ is a bounded Lipschitz continuous functional, it follows that for every $s\ge0$, 
 \begin{align*}
 |\ES[F(Y_{.,1})]-\ES[F(Y_{.,2})]|\le C\sum_{i=1}^{2} \ES[\sup_{t\in[s,s+T]}| Y_{t,i}-X_{t,i}^x|\wedge 1].
 \end{align*} 
 In order to obtain \eqref{egalloi}, it is now enough to prove that 
 
$$ \sup_{t\ge s}| Y_{t,i}-X_{t,i}^x|\xrn{s\nrn}0\quad a.s.,\quad i=1,2.$$
 Set $V_{t}^i= |Y_{t,i}-X_{t,i}^x|^2$. We have:
 $$dV_{t}^i=2\psg b(Y_{t,i})-b(X_{t,i}^x),Y_{t,i}-X_{t,i}^x\psd.$$
Thus, it follows from $\mathbf{(H_3)}$ with $\beta =0$ and from the Gronwall lemma that,
$$|Y_{t,i}-X_{t,i}^x|^2\le (Y_0-x)^2\exp(-2\alpha t).$$
Therefore,
$$ \sup_{t\ge s}| Y_{t,i}-X_{t,i}^x|^2\le (Y_0-x)^2\exp(-2\alpha s) \xrn{s\nrn}0\quad a.s.,\quad i=1,2.$$
This concludes the proof of the uniqueness for the distribution of a stationary solution to \eqref{fractionalSDE}. For Equation $\mathbf{(E_\gamma)}$, the proof is a straightforward adaptation of the previous one. Details are left to the reader.
\end{proof} 
\begin{prop}\label{propid2} Assume $\mathbf{(H_1)}$ and $\mathbf{(H_3)}$. Then, there exists $\gamma_0>0$ such that $a.s.$
$({\cal P}^{(\infty,\gamma)}(\omega,d\alpha))_{\gamma\in(0,\gamma_0)}$ is relatively compact for the topology of uniform convergence on compact sets.
Furthermore, any weak limit of  $({\cal P}^{(\infty,\gamma)}(\omega,d\alpha))_{\gamma\in(0,\gamma_0)}$ (when $\gamma\rightarrow0$) is the distribution of a stationary solution to SDE \eqref{fractionalSDE}.
\end{prop}
\begin{proof}
\vspace{0.1cm}
\noindent\textbf{Step 1:} A.s. tightness of $({\cal P}^{(\infty,\gamma)}(\omega,d\alpha))$: For $\omega\in\Omega$, we recall that $Y^{(\infty,\gamma)}$ is a càdlàg  process with distribution ${\cal P}^{(\infty,\gamma)}(\omega,d\alpha)$.
According to Theorem VI.3.26 of \cite{jacodshiryaev}, we have to show the two following points:
\begin{itemize}
\item{} For every $T>0$, there exists $\gamma_0>0$ such that
\begin{equation}\label{point1}
\limsup_{K\rightarrow+\infty}\sup_{\gamma\in(0,\gamma_0]}\PE(\sup_{t\in[0,T]}|Y^{(\infty,\gamma)}_t|>K)=0.
\end{equation}
\item{} For every positive $T$, $\varepsilon$ and $\eta$, there exist $\delta>0$ and $\gamma_0>0$  such that for every $\gamma\le\gamma_0$,
\begin{equation}\label{point2}
\PE(\sup_{|t-s|\le\delta,0\le s\le t\le T}|Y^{(\infty,\gamma)}_t-Y^{(\infty,\gamma)}_s|\ge\varepsilon)
\le\eta.
\end{equation}
\end{itemize}
First, we focus on \eqref{point1}. Let $K>0$. By Proposition \ref{prop2}, we have:
\begin{align*}
 \PE(\sup_{t\in[0,T]}|Y^{(\infty,\gamma)}_t|>K)&\le \PE\left(|Y^{(\infty,\gamma)}_0|+\int_0^T |b(Y^{(\infty,\gamma)}_s)|ds+\sup_{t\in[0,T]}| \bar{Z}^{\gamma}_t|>{K}\right).
\end{align*}
Using the Markov inequality, it follows that
$$\PE(\sup_{t\in[0,T]}|Y^{(\infty,\gamma)}_t|>K)\le\frac{1}{K}\left(\ES[|Y^{(\infty,\gamma)}_0|]+
C T\sup_{n\in\{0,\ldots,[T/\gamma]\}}\ES[|Y^{(\infty,\gamma)}_{n\gamma}|]+\ES[\sup_{t\in[0,T]}|Z_t|]\right)$$
where $\|\sigma\|=\sup\{|\sigma x|/|x|,x\in\ER^d\}.$ Now, since $(Y_{n\gamma})$ is a stationary sequence and $\sup_{t\in[0,T]}|Z_t|$ is integrable (see Proposition \ref{propcontrolesup}), one obtains:
$$\PE(\sup_{t\in[0,T]}|Y^{(\infty,\gamma)}_t|>K)\le\frac{C}{K}\left(1+\ES[|Y^{(\infty,\gamma)}_0|]\right),$$
where $C$ does not depend on $\gamma$. Finally, the first point follows from Lemma \ref{lemme4}.\\

\noindent Let us now prove \eqref{point2}. In fact, using for instance proof of Theorem 8.3 of \cite{billingsley}, it is enough to show that for every
positive  $\varepsilon,\eta$ and  $T$, there exist $\delta>0$ and
$\gamma_0>0$ such that for every $\gamma\le\gamma_0$:
\begin{equation}\label{ty}
\frac{1}{\delta}\mathbb{P}(\sup_{t\le s\le
t+\delta}|Y^{(\infty,\gamma)}_t-Y^{(\infty,\gamma)}_s|\ge\varepsilon)\le\eta\qquad\forall
\gamma\le\gamma_0\quad\textnormal{and}\quad 0\le  t\le T.
\end{equation}
By the Markov inequality, we have for every $p\ge1$:
\begin{align}
\label{inf1}\mathbb{P}\left(\sup_{t\le s\le
t+\delta}|Y^{(\infty,\gamma)}_t-Y^{(\infty,\gamma)}_s|\ge\varepsilon\right)&\le\left(\frac{2}{\varepsilon}\right)^2\ES\left[\left(\int_{\underline{t}_\gamma}^{\underline{t+\delta}_\gamma}|b(Y^{(\infty,\gamma)}_s)|ds\right)^2\right]\\&+\left(\frac{2}{\varepsilon}\right)^p\ES\left[\sup_{s\in[t,t+\delta]}|\bar{Z}_s^{\gamma}-\bar{Z}_t^{\gamma}|^p\right]\nonumber
\end{align}
On the one hand,
\begin{align*} 
\ES\left[\left(\int_t^{t+\delta}|b(Y^{(\infty,\gamma)}_s)|ds\right)^2\right]&\le \ES\left[\left(\sum_{k=[t/\gamma]}^{[(t+\delta)/\gamma]}\sqrt{\gamma}(\sqrt{\gamma}|b(Y^{(\infty,\gamma)}_{n\gamma})|)\right)^2\right]\\&\le\ES\left[\left(\sum_{k=[t/\gamma]}^{[(t+\delta)/\gamma]}\gamma\right)\left(\sum_{k=[t/\gamma]}^{[(t+\delta)/\gamma]}\gamma|b(Y^{(\infty,\gamma)}_{n\gamma})|^2\right)\right]
\end{align*}
thanks to the Cauchy-Schwarz inequality. Now, when $\gamma$ is sufficiently small $$\sum_{k=[t/\gamma]}^{[(t+\delta)/\gamma]}\gamma\le 2\delta.$$ Therefore, using also the fact that $b$ has sublinear growth yields:
\begin{equation}
\ES\left[\left(\int_t^{t+\delta}|b(Y^{(\infty,\gamma)}_s)|ds\right)^2\right]\le C\delta^2(1+\sup_{k\in\{[\frac{t}{\gamma}],\ldots,[\frac{(t+\delta)}{\gamma}]\}}\ES[|Y^{(\infty,\gamma)}_{k\gamma}|^2])\le C\delta^2(1+\ES[|Y^{(\infty,\gamma)}_{0}|^2])\label{inf2}
\end{equation}
thanks to the stationarity of $(Y^{(\infty,\gamma)}_{n\gamma})_{n\ge0}$.\\
\noindent On the other hand, we deduce from the stationarity of the increments of $(Z_t)_{t\ge0}$ that
\begin{equation*}
\ES\left[\sup_{s\in[t,t+\delta]}|\bar{Z}_s^{\gamma}-\bar{Z}_t^{\gamma}|^p\right]\le \ES\left[\sup_{s\in[t,t+\delta]}|Z_s-Z_t|^p\right]
\le \ES[\sup_{s\in[0,\delta]}|Z_s|^p].
\end{equation*}
Thus, by  Proposition \ref{propcontrolesup} (see Appendix), for sufficiently large $p$, 
\begin{equation}\label{inf3}
\ES\left[\sup_{s\in[t,t+\delta]}|\bar{Z}_s^{\gamma}-\bar{Z}_t^{\gamma}|^p\right]\le C\delta^{1+\rho}
\end{equation}
where $\rho$ is a positive number.\\
Then, the combination of \eqref{inf1}, \eqref{inf2} and \eqref{inf3} yields for sufficiently small $\gamma$:
$$\mathbb{P}\left(\sup_{t\le s\le
t+\delta}|Y^{(\infty,\gamma)}_t-Y^{(\infty,\gamma)}_s|\ge\varepsilon\right)\le C\delta^{2\wedge(1+\rho)}$$
and \eqref{point2} follows from Lemma \ref{lemme4}.\\

\noindent \textbf{Step 2:} We want to show that, $a.s$, any weak limit ${\cal P}(\omega,d\alpha)$ of $({\cal P}^{(\infty,\gamma)}(\omega,d\alpha))_{\gamma}$ when $\gamma\rightarrow0$ (for the uniform convergence topology)
is the distribution of a stationary process. Let $f:\ER^r\rightarrow\ER$ be a bounded continuous function and let $t>0$ and $t_1,\ldots, t_r$ such that $0\le t_1<\ldots<t_r.$  Denoting by $(Y_t)_{t\ge0}$ a process with distribution 
${\cal P}(\omega,d\alpha)$, we have to show that:
\begin{equation}\label{statcond}
\ES[f(Y_{t_1+t},\ldots,Y_{t_r+t})]=\ES[f(Y_{t_1},\ldots,Y_{t_r})].
\end{equation}
First, since ${\cal P}(\omega,d\alpha)$ is a weak limit of $({\cal P}^{(\infty,\gamma)}(\omega,d\alpha))_{\gamma}$, there exist some sequences  $(\gamma_n)_{n\ge0}$ and $(Y^{(\infty,\gamma_n)})_{n\ge0}$  such that 
${\cal L}(Y^{(\infty,\gamma_n)})={\cal P}^{(\infty,\gamma_n)}(\omega,d\alpha)$ and $(Y^{(\infty,\gamma_n)}_t)$ converges weakly to $(Y_t)$ for the weak topology induced by the uniform convergence topology on compact sets on $\mathbb{D}(\ER_+,\ER^d)$.
In particular, 
\begin{align}\label{convstat}
&\ES[f(Y^{(\infty,\gamma_n)}_{t_1},\ldots,Y^{(\infty,\gamma_n)}_{t_r})]\xrn{n\nrn}\ES[f(Y_{t_1},\ldots,Y_{t_r})]\quad\textnormal{and,}\\
&\ES[f(Y^{(\infty,\gamma_n)}_{t_1+\underline{t}_{\gamma_n}},\ldots,Y^{(\infty,\gamma_n)}_{t_r+\underline{t}_{\gamma_n}})]\xrn{n\nrn}\ES[f(Y_{t_1+t},\ldots,Y_{t_r+t})].
\end{align}
since $\underline{t}_{\gamma_n}:=\gamma_n[t/\gamma_n]\rightarrow t$ when $n\rightarrow+\infty$. Now, by Proposition \ref{prop2},
$$\ES[f(Y^{(\infty,\gamma_n)}_{t_1},\ldots,Y^{(\infty,\gamma_n)}_{t_r})]=\ES[f(Y^{(\infty,\gamma_n)}_{t_1+\underline{t}_{\gamma_n}},\ldots,Y^{(\infty,\gamma_n)}_{t_r+\underline{t}_{\gamma_n}})]\quad\forall n\ge1.$$
\eqref{statcond} follows.\\

\noindent \textbf{Step 3:} Let $\Phi:\mathbb{D}(\ER_+,\ER^d)\rightarrow\mathbb{D}(\ER_+,\ER^d)$ be defined by
$$(\Phi(\alpha))_t=\alpha_t-\alpha_0-\int_0^t b(\alpha_s)ds.$$
With the notations of Step 2, we want to show that $Y:=(Y_t)_{t\ge0}$ is a solution to \eqref{fractionalSDE}, $i.e.$ that $\Phi(Y)$
is equal in law to $ Z:=( Z_t)_{t\ge0}$. Let $(\gamma_n)$ and $(Y^{(\infty,\gamma_n)})_{n\ge0}$ be defined as in Step 2. Then, since $\Phi$ is continuous for the uniform convergence topology on compact sets, 
\begin{equation}\label{uyi1}
\Phi(Y^{(\infty,\gamma_n)})\overset{n\rightarrow+\infty}{\Longrightarrow} \Phi(Y),
\end{equation}
for the weak topology induced by the uniform convergence topology on compact sets. Therefore, we have to prove that
\begin{equation}\label{identfrac}
\Phi(Y^{(\infty,\gamma_n)})\overset{n\rightarrow+\infty}{\Longrightarrow} Z,
\end{equation}
for this topology. With the notations of Proposition \ref{prop2},
\begin{equation}\label{uyi2}
\Phi(Y^{(\infty,\gamma_n)})=N^{(\infty,\gamma_n)}+R^{\gamma_n}\quad\textnormal{where}\quad R_t^{\gamma_n}=-\int_{\underline{t}_{\gamma_n}}^t b(Y_s^{(\infty,\gamma_n)})ds.
\end{equation}
First, since $b$ is sublinear and  $t-\underline{t}_{\gamma_n}\le\gamma_n$, we have for every $T>0$:
$$|R^{\gamma_n}_t|\le C\gamma_n(1+\sup_{t\in[0,T]}|Y_t^{(\infty,{\gamma_n})}|)\quad\forall t\in[0,T].$$
Now, in Step 1, we showed that   $(\sup_{t\in[0,T]}|Y_t^{(\infty,\gamma)}|)_{\gamma\in(0,\gamma_0)}$ is tight on $\ER$. It follows easily that
$$\sup_{t\in[0,T]}|R^{\gamma_n}_t|\xrn{n\nrn}0\quad \textnormal{in probability }\forall T>0.$$
Therefore, one derives from \eqref{uyi1} and \eqref{uyi2},
\begin{equation*}
N^{(\infty,\gamma_n)}\overset{n\rightarrow+\infty}{\Longrightarrow} \Phi(Y).
\end{equation*}
Then, it follows from Proposition \ref{prop2} that 
$N^{(\infty,\gamma_n)}$ is a convergent sequence of Gaussian processes such that $N^{(\infty,\gamma_n)}\overset{\cal L}{=} \bar{Z}^{\gamma_n}$. 
This implies the finite-dimensional convergence to $(Z_t)_{t\ge0}$ and concludes the proof.
\end{proof}
\section{Properties of the stationary solution}\label{sect-properties}
In this section,we give some properties of the stationary solution. In Subsection \ref{sub-depend}, we prove  that the random initial value of a stationary solution can only be independent of a Gaussian noise with independent increments. In Subsection \ref{sub-independ}, we prove that however, an independence property holds between the past of the stationary solution and the future of the so-called innovation process. 
\subsection{Dependence between $X_0$ and $Z$}\label{sub-depend}
\begin{prop}\label{dependenceprop} Let $X_0$ and $(Z_t)_{t\ge0}$ denote the random initial value and the driving process of a stationary solution to  \eqref{fractionalSDE}. Then, if $X_0$ is independent of $(Z_t)_{t\ge0}$, then $(Z_t)_{t\ge0}$ has independent increments. As a consequence, $Z=QW$ where $W$ is a standard $d$-dimensional Brownian Motion and $Q$ is a deterministic matrix.
\end{prop}

\begin{proof} Let  $X:=(X_t)_{t\ge0}$ be a stationary solution  to \eqref{fractionalSDE} and assume that  $X_0$ is independent of $Z:=(Z_t)_{t\ge0}$.  First, note that for every $t\ge0$, $Z_{t+\,\cdot}-Z_t=\psi(X_{t+\,\cdot})$ where $\psi:{\cal C}(\ER_+,\ER^d)\rightarrow {\cal C}(\ER_+,\ER^d)$ is defined for every $\alpha\in{\cal C}(\ER_+,\ER^d)$ by
$$ \psi(\alpha)_t=\alpha_t-\alpha_0-\int_0^t b(\alpha_s)ds\quad \forall t\ge0,$$
is continuous. Then, since $(X_t)$ is stationary, it follows  that for every bounded continuous functional  $F$ (for the topology of uniform convergence on compact sets), for every $s\ge0$, for every $t\ge0$, and every bounded continuous function $f:\ER^d\rightarrow\ER$,
\begin{align*}
\ES[\left(f(X_0)-\ES[f(X_0)]\right) F(Z_{s+u}-Z_s&,u\ge0)]=\ES[\left(f(X_t)-\ES[f(X_t)]\right) F(\psi(X_{t+s+\,\cdot}))]\\&=\ES[\left(f(X_t)-\ES[f(X_t)]\right) F(Z_{t+s+u}-Z_{t+s},u\ge0)].
\end{align*}
Since $X_0$ is independent of $Z$,  
$$\ES[\left(f(X_0)-\ES[f(X_0)]\right) F(Z_{s+u}-Z_s,u\ge0)]=0.$$
This implies that
$$\ES[f(X_t)F(Z_{t+s+u}-Z_t,u\ge0)]=\ES[f(X_t)]\ES[F(Z_{t+s+u}-Z_{t+s},u\ge0)].$$
One deduces that for every $s,t\ge0$, such that $0\le s\le t$, $X_s$ is independent of $(Z_{t+u}-Z_t)_{u\ge0}$. 
As a consequence, for every positive $u,t$, for every $i,j\in\{1,\ldots,d\}$,
\begin{align*}
\ES[Z_t^i(Z_{t+u}^j-&Z_{t}^j)]\\
&=\left(\ES[X_t^i(Z_{t+u}^j-Z_{t}^j)]-\int_0^t\ES[b^i(X_v)(Z_{t+u}^j-Z_{t}^j)]dv-\ES[X_0^i(Z_{t+u}^j-Z_{t}^j)]\right),\\
&=\left(\ES[X_t^i]\ES[Z_{t+u}^j-Z_{t}^j]-\int_0^t\ES[b^i(X_v)]\ES[Z_{t+u}^j-Z_{t}^j]dv-\ES[X_0^i]\ES[Z_{t+u}^j-Z_{t}^j]\right),\\
&=0.
\end{align*}
Since $Z$ is a centered Gaussian process, it clearly implies that $Z$ has independent increments.
\end{proof}
\subsection{Independency to the future of the underlying innovation process}\label{sub-independ}
In order to ensure a physical sense to the stationary solutions built with our discrete approach, it may be important to check that for every $t\ge0$, $\sigma(X_s,0\le s\le t)$ does not depend on the \textit{innovation} generated by the \textit{future} of the Gaussian process $Z$ after $t$ (see below for a precise definition). In the simple case where $Z$ is a Brownian motion (whose increments are independent), it is natural to define the innovation after $t$ as the $\sigma$-field generated by the increments of the Brownian motion after $t$. Then, a stationary solution of a SDE driven by $Z$ is  meaningless if $\sigma(X_s,0\le s\le t)$ is not independent of the increments of $Z$ after $t$. This is the case for $(X_t)_{t\ge0}$ defined by
$$X_t=-\int_t^{+\infty}e^{t-s}dZ_s$$
which is a stationary solution (in the sense of Definition \ref{stat-sol}) to $dX_t=X_tdt+dZ_t$ but whose initial value $X_0=-\int_0^{+\infty} e^{-s}dZ_s$ 
depends on all the future of $Z$ (Note that $(X_t)_{t\ge0}$ does not satisfy Assumption $\mathbf{(H_2)}$). Thus, the aim of this section is to show that our construction of stationary solutions as weak limits of ergodic Euler schemes does not generate such a type of stationary solutions.\\
This section is divided in two parts. We focus successively on the discrete case, $i.e.$ on the stationary solutions to $\mathbf{(E_\gamma)}$ and on the continuous case, $i.e.$ to the stationary solutions to \eqref{fractionalSDE}. Oppositely to the rest of the paper, we will need in the two following parts to introduce series or integral representations in order to define the innovation of $(\Delta_n)$ and $Z$ respectively. We will also assume that $d=1$.\\

\noindent \textbf{The discrete case.} Assume that $(\Delta_n)_{n\ge1}$ is a purely non-deterministic sequence. Then, by Theorem 3.2 of \cite{hida-hitsuda}, there exists a sequence of real numbers denoted by $(a_{k,\gamma})_{k\ge0}$ such that $(\Delta_n)_{n\ge1}$ admits the following representation:
\begin{equation}\label{rep-discrete}
\Delta_n=\sum_{k=0}^{+\infty} a_{k,\gamma} \xi_{n-k}\quad a.s., \quad\forall n\ge1,
\end{equation}
where  $(\xi_k)_{k\in\mathbb{Z}}$, is a sequence of $i.i.d.$ centered real-valued Gaussian variables such that ${\rm Var}(\xi_1)=1$. The sequence $(\xi_k)_{k\in\mathbb{Z}}$ is then called the underlying innovation process associated with $(\Delta_n)_{n\ge1}$. Under the assumptions of Theorem \ref{principal}, we know that for $\gamma$ small enough, there exists a stationary distribution $\bar{{\cal P}}^{\infty,\gamma}$ on $\ER^\mathbb{N}$ to the recursive equation 
\begin{equation}\label{rec-eq-2}
\bar{X}_{(n+1)\gamma}-\bar{X}_{n\gamma}=\gamma b(\bar{X}_{n\gamma})+\Delta_{n+1}\quad\forall n\ge0.
\end{equation}
By Lemma \ref{lemme-verif-const} in the Appendix, this stationary distribution $\bar{\cal P}^{\infty,\gamma}$ can be realized 
as $ \kappa_1(\mu_1),$  where $\mu_1$ denotes the probability measure on $\ER\times\ER^{\mathbb{Z}}$ defined 
by~\eqref{eq:mu1} such that its projection on the second coordinate is ${\cal L}((\xi_n)_{n\in\mathbb{Z}}),$ and 
where $ \kappa_1$ is defined in~\eqref{kappa1}. From now on, we denote this stationary solution by  $({\bar X}^{(\infty,\gamma)}_{n\gamma})_{n\ge0}$.  
In the next proposition, we show that for $\gamma$ small enough, the past of $(\bar{X}^{(\infty,\gamma)}_{n\gamma})_{n\ge0}$ is independent of the future of this innovation process.
\begin{prop}\label{innovation-discrete}
Assume $\mathbf{(H_1)}$ and $\mathbf{(H_{3,0})}$. Let $\gamma_0>0$ such that the conclusions of Theorem \ref{principal} hold for $\gamma\le \gamma_0$ and denote by $\bar{\cal P}^{\infty,\gamma}$ the unique stationary distribution on $\ER^{\mathbb{N}}$ to \eqref{rec-eq-2}. Let $({\bar X}^{\infty,\gamma}_{n\gamma})_{n\ge0}$ denote the realization of $\bar{\cal P}^{\infty,\gamma}$ on $\ER\times\ER^{\mathbb{Z}}$ defined previously. Then, for $\gamma$ small enough, for every  $n\ge 0$, $\sigma(\bar{X}^{(\infty,\gamma)}_0,\ldots,\bar{X}^{(\infty,\gamma)}_{n\gamma})$ is independent to $\sigma(\xi_{k},k\ge n+1)$. 
\end{prop}
\begin{Remarque} For the sake of simplicity, we stated Proposition \ref{innovation-discrete} under $\mathbf{(H_{3,0})}$
which ensures uniqueness of the distribution of the stationary solution. However, adapting the arguments of the proof to subsequences, it could be possible to  show that, under $\mathbf{(H_1)}$ and $\mathbf{(H_{2})}$ only, there exists $\gamma_0>0$ such that for every $\gamma\le \gamma_0$, $a.s.$, every weak limit of ${\cal P}^{(n,\gamma)}(\omega,d\alpha)$ (as $n\rightarrow+\infty$) is the distribution of a stationary solution to~$\mathbf{(E_\gamma)}$   that satisfies the preceding independence property. Note that the same type of remark holds for Proposition \ref{innovation-continue} below.
\end{Remarque}
\begin{proof} 
It is enough to prove that for every integers $N_1$ and $N_2$ such that $N_2>N_1$, for every bounded continuous functions $H_1$ and $H_2$,
\begin{align}
\ES[H_1(\bar{X}_{0}^{(\infty,\gamma)},\ldots, \bar{X}_{\gamma N_1}^{(\infty,\gamma)}) H_2(&\xi_{N_1+1},\ldots,\xi_{N_2})]\nonumber\\&
=\ES[H_1(\bar{X}_{0}^{(\infty,\gamma)},\ldots, \bar{X}_{\gamma N_1}^{(\infty,\gamma)})]\ES[H_2(\xi_{N_1+1},\ldots,\xi_{N_2})].\label{egaamontrer1}
\end{align} 
In order to make use of our previous convergence results, we first write  $(\bar{X}_{k\gamma}^{(\infty,\gamma)},\xi_{k})_{k\ge 1}$ as a function of a stationary sequence $(\bar{X}_{\gamma k}^{(\infty,\gamma)},\bar{S}_{\gamma k}^{(\infty,\gamma)})_{k\ge0}$: we denote by $\mathbf{(\tilde{E}_\gamma)}$ the recursive equation on $\ER^{2}$ defined  for every $n\ge0$ by 
\begin{equation*}
\begin{cases}& \bar{X}_{(n+1)\gamma}-\bar{X}_{n\gamma}=\gamma b(\bar{X}_{n\gamma})+\sigma \Delta_{n+1}\\
&\bar{S}_{(n+1)\gamma}-\bar{S}_{n\gamma}=-\gamma\bar{S}_{n\gamma}+\xi_{n+1}.
\end{cases}
\end{equation*} 
We will also denote by $({\bar{X}}_{n\gamma}^x,{\bar{S}}_{n\gamma}^s)$ a solution to $\mathbf{(\tilde{E}_\gamma)}$ starting from a deterministic point $(x,s)$, by $(({\bar{X}}_{t}^x,{\bar{S}}_{t}^s))_{t\ge0}$ the induced stepwise constant process and by $(({\bar{X}}_{\gamma k+.}^s,{\bar{S}}_{\gamma k +.}^s)$ the $\gamma k$-shifted process. Now, one observes that Assumptions $\mathbf{(H_1)}$ and $\mathbf{(H_{3,0})}$ are satisfied on $\ER^{2}$ for $(X_t,S_t)$ with $\tilde{b}(x,s)=(b(x),-s)$ and $\tilde{\Delta}_n=(\Delta_n,\xi_n)$. 
Hence, by Theorem \ref{principal}.1 and the uniqueness induced by Assumption $\mathbf{(H_{3,0})}$, there exists $\gamma_0>0$ such that for every $\gamma\le \gamma_0$, $a.s.$, for every bounded continuous functional $F:\mathbb{D}(\ER_+,\ER^d)\rightarrow\ER$, 
\begin{equation}\label{conv-couple}
\frac{1}{n}\sum_{k=1}^n F\left(({\bar{X}}_{\gamma (k-1)+.}^x,{\bar{S}}_{\gamma (k-1) +.}^s\right)
\xrightarrow{n\rightarrow+\infty}\tilde{\cal P}^{(\infty,\gamma)}(F)
\end{equation}
where $\tilde{\cal P}^{(\infty,\gamma)}$ is the unique distribution of a (discretely) stationary solution that we denote by $(\bar{X}^{(\infty,\gamma)},\bar{S}^{(\infty,\gamma)})$. Thus, on the one hand, since $(\xi_k)_{k\in\{N_1+1,\ldots,N_2\}}$ is clearly a continuous function $G$ of
$ (\bar{S}^{(\infty,\gamma)}_k)_{0\le k\le N_2}$, we deduce that $a.s.$,
\begin{align}
 \ES[H_1(\bar{X}_{0}^{(\infty,\gamma)}&,\ldots, \bar{X}_{\gamma N_1}^{(\infty,\gamma)}) H_2(\xi_{N_1+1},\ldots,\xi_{N_2})]\nonumber\\
 &=
 \lim_{n\rightarrow+\infty}\frac{1}{n}\sum_{k=1}^n H_1(\bar{X}_{\gamma(k-1)}^x,\ldots,\bar{X}_{\gamma(k-1+N_1)}^x)H_2(G(\bar{S}_{\gamma(k-1)}^s),\ldots,\bar{S}_{\gamma(k-1+N_2)}^s))\nonumber\\
 &=\lim_{n\rightarrow+\infty}\frac{1}{n}\sum_{k=1}^n H_1(\bar{X}_{\gamma(k-1)}^x,\ldots,\bar{X}_{\gamma(k-1+N_1)}^x)H_2(\xi_{k-1+N_1},\ldots,\xi_{k-1+N_2}).\label{eq:cc1}
 \end{align}
On the other hand, by \eqref{conv-couple} and the fact that $(\xi_k)_{k\ge1}$ is a stationary sequence, we also have
\begin{align}\nonumber
 \ES[&H_1(\bar{X}_{0}^{(\infty,\gamma)},\ldots, \bar{X}_{\gamma N_1}^{(\infty,\gamma)})\ES[ H_2(\xi_{N_1+1},\ldots,\xi_{N_2})]\\
 &=\lim_{n\rightarrow+\infty}\frac{1}{n}\sum_{k=1}^n H_1(\bar{X}_{\gamma(k-1)}^x,\ldots,\bar{X}_{\gamma(k-1+N_1)}^x)\ES[H_2(\xi_{k-1+N_1},\ldots,\xi_{k-1+N_2})]\quad a.s.\label{eq:cc2}
 \end{align}
 Setting 
 $$\zeta(k,N_1,x)=H_1(\bar{X}_{\gamma k}^x,\ldots,\bar{X}_{\gamma(k+N_1)}^x)\quad\textnormal{and} \quad\Lambda(k,N_2)=H_2(\xi_{k+N_1},\ldots,\xi_{k+N_2}),$$
we deduce from \eqref{eq:cc1} and \eqref{eq:cc2} that \eqref{egaamontrer1} is true if 
\begin{equation}\label{eq:martcc}
\frac{1}{n}\sum_{k=1}^n\zeta(k-1,N_1,x)\left(\Lambda(k-1,N_2)-\ES[\Lambda(k-1,N_2)]\right)\xrn{n\nrn}0\quad a.s.
\end{equation}  
We use a martingale argument. Set ${\cal H}_{\ell}=\sigma(\xi_k,k\in\mathbb{Z},k\le \ell)$. Since $\Lambda(k,N_2)$ is a ${\cal H}_{k+N_2}$-measurable random variable independent of  ${\cal H}_{k+N_1}$ and that $\zeta(k,N_1,x)$ is ${\cal H}_{k+N_1}$-measurable, we can write:
$$\sum_{k=1}^n\frac{\zeta(k-1,N_1,x)}{k}\left(\Lambda(k-1,N_2)-\ES[\Lambda(k-1,N_2)]\right)=\sum_{\ell=N_1}^{N_2-1} M_n^\ell$$
where for every $\ell\in\{N_1+1,\ldots,N_2\}$, $(M_n^\ell)_{n\ge1}$ is a centered $({\cal H}_{n-1+\ell})_{n\ge1}$-adapted martingale defined by:
$$M_n^\ell=\sum_{k=1}^n\frac{\zeta(k-1,N_1,x)}{k}\left(\ES[\Lambda(k-1,N_2)/{\cal H}_{k+\ell}]-\ES[\Lambda(k-1,N_2)/{\cal H}_{k+\ell-1}]\right).$$
For every $\ell\in\{N_1+1,\ldots,N_2\}$,  $(M_n^\ell)_{n\ge1}$ is clearly bounded in $L^2$ since $H_1$ and $H_2$ are bounded. Thus, setting $a_k=1/k$ and 
$$b_k(\omega)=\zeta(k-1,N_1,x)\left(\ES[\Lambda(k-1,N_2)/{\cal H}_{k+\ell}]-\ES[\Lambda(k-1,N_2)/{\cal H}_{k+\ell-1}]\right),$$
we deduce that the serie $\sum a_k b_k(\omega)$ is $a.s.$ convergent and \eqref{eq:martcc} follows from the Kronecker lemma.
\end{proof}

\noindent\textbf{The continuous case.}  We assume in this part that, $Z$ admits the following representation:
\begin{equation}\label{rep-movingaverage}
Z_t=\int_{-\infty}^t f_t(s)dW_s \quad a.s. \quad \forall t\ge0.\end{equation}
where  $(W_t)_{t\in\ER}$ is a two-sided standard Brownian motion 
such that $ W_0=0$ 
and for every $t\in\ER$, $f_t\in L^2(\ER,\ER)$. 

Following Lemma \ref{lemme-verif-const} in the Appendix, the distribution of  a stationary solution $(X_t)_{t\ge0}$ 
to~\eqref{fractionalSDE} can be realized as $ \kappa_2(\mu_2),$  where $\mu_2$ denotes the probability measure on $\ER\times{\cal C}(\ER,\ER)$  defined  in Lemma \ref{lemme-verif-const} such that its projection on the second coordinate is ${\cal L}((W_t)_{t\in\ER}),$ and 
where $ \kappa_2$ is defined in~\eqref{kappa2}. In the next proposition, we show that the past of the stationary solution $(X_t)_{t\ge0}$  is independent to the future of the underlying innovation process of $Z$. 

\begin{prop}\label{innovation-continue} Assume $\mathbf{(H_1)}$ and $\mathbf{(H_{3,0})}$. Assume that  the representation \eqref{rep-movingaverage} holds for $Z$. Let ${\cal P}^\infty$ denote the unique distribution on ${\cal C}(\ER_+,\ER)$ of a stationary solution. Let $(X_t)_{t\ge0}$ denote the realization on $\ER\times{\cal C}(\ER,\ER)$ of ${\cal P}^\infty$  defined in Lemma \ref{lemme-verif-const}. Then, for every $t\ge0$, $\sigma(X_s,0\le s\le t)$ is independent of $\sigma(W_{s+t}-W_t,s\ge0)$.
\end{prop}
\begin{proof} Let $t\ge0$. It is enough to show that for every $T>0$,  for every bounded Lipschitz continuous functionals $H_1:{\cal C}([0,t],\ER^d)\mapsto\ER$ and $H_2:{\cal C}([0,T],\ER^d)\mapsto\ER$
\begin{align}
\ES[H_1(X_s,0\le s\le t) H_2&(W_{t+s}-W_t,0\le s\le T)]\nonumber\\&
=\ES[H_1(X_s,0\le s\le t)]\ES[H_2(W_{t+s}-W_t,0\le s\le T)].\label{egaamontrer}
\end{align} 
As in the proof of Proposition \ref{innovation-discrete}, we introduce
$(X_t,S_t)$ that is a stationary solution to 
\begin{equation*}
\begin{cases}& dX_t= b(X_t)dt+ dZ_t\\
&dS_t=-S_t dt+dW_t
\end{cases}
\end{equation*}
where $W$ denotes the underlying innovation process of the representation \eqref{rep-movingaverage}. Following the second part of the Appendix (see Lemma \ref{lemme-verif-const}), we can assume that $(X_t,S_t)$ is built on $\ER^2\times{\cal C}(\ER,\ER)$. Assumptions $\mathbf{(H_1)}$ and $\mathbf{(H_{3,0})}$ are satisfied on $\ER^{2}$ for $(X_t,S_t)$ with $\tilde{b}(x,s)=(b(x),-s)$ and $\tilde{Z}_t=(Z_t,W_t)$. Let $\gamma_0$ and $\gamma_1\in(0,\gamma_0)$ such that Theorem \ref{principal} holds for every $\gamma\le\gamma_1$. Denote by $\tilde{{\cal P}}^{\infty,\gamma}$ the (unique) stationary distribution of $(\bar{E}_\gamma)$. Denote by $(\bar{X}_t^{\infty,\gamma},\bar{S}_t^{\infty,\gamma})_{t\ge0}$ the stationary solution built on $\ER^2\times{\cal C}(\ER,\ER)$. Using that
$$(W_{t+s}-W_t,0\le s\le T)=(\tilde{\phi}(S_{s+t}))_{s\in[0,T]}$$
where $\tilde{\phi}(\alpha)$ is defined by
$$\tilde{\phi}(\alpha)_s=\alpha_s-\alpha_0+\int_0^s\alpha_u du,$$
we deduce from Proposition \ref{propid2} and from the continuity of $\tilde{\phi}$ for the uniform topology (on compact sets) that
\begin{align*}
\ES[H_1(X_s,0\le s\le t)H_2(W_{t+s}-&W_t,0\le s\le T)]\\&=\lim_{\gamma\rightarrow0}\ES[H_1(\bar{X}_s^{\infty,\gamma},0\le s\le t)H_2((\tilde{\phi}(\bar{S}_{s+t}^{\infty,\gamma})_{s\in[0,T]})].
\end{align*}
Now, setting $\xi_n=W_{n\gamma}-W_{(n-1)\gamma}$, we derive from the definition of the discretized equation that
$$(\tilde{\phi}(\bar{S}_{.+t}^{\infty,\gamma}))_s=\sum_{k=[\frac{t}{\gamma}]+1}^{[\frac{t+s}{\gamma}]}\xi_k+R_\omega(\gamma,t,s)$$
where
$$|R_\omega(\gamma,t,s)|=|(t+s-\underline{t+s}_\gamma)\bar{S}_{\underline{t+s}_\gamma}-(t-\underline{t}_\gamma)\bar{S}_{\underline{t}_\gamma}|\le 2\gamma \sup_{u\in[0,t+T]}|\bar{S}^{\infty,\gamma}_u|.$$
Since $(\bar{S}^{\infty,\gamma})_{\gamma\in(0,\gamma_1]}$ is tight for the uniform topology on compact sets, it follows that
$$\sup_{s\in[0,T]}|R_\omega(\gamma,t,s)|\xrn{\gamma\rightarrow0}0\quad\textnormal{in probability.}$$
Thus, using that $H_2$ is Lipschitz continuous, we deduce that
\begin{align*}
\ES[H_1(X_s,0\le s\le t)H_2(W_{t+s}-&W_t,0\le s\le T)]\\&=\lim_{\gamma\rightarrow0}\ES[H_1(\bar{X}_s^{\infty,\gamma},0\le s\le t)H_2((\sum_{k=[\frac{t}{\gamma}]+1}^{[\frac{t+s}{\gamma}]}\xi_k)_{s\in[0,T]})].
\end{align*}
\noindent As well,
$$\ES[H_1(X_s,0\le s\le t)]=\lim_{\gamma\rightarrow0}\ES[H_1(\bar{X}_s^{\infty,\gamma},0\le s\le t)]$$ 
and,
$$\ES[H_2(W_{t+s}-W_t,0\le s\le T)]=\lim_{\gamma\rightarrow0}\ES[H_2((\sum_{k=[\frac{t}{\gamma}]+1}^{[\frac{t+s}{\gamma}]}\xi_k)_{s\in[0,T]})].$$
By the previous convergences, it is now enough to show that for $\gamma$ sufficiently small, 
\begin{align*}
\ES[H_1(\bar{X}_s^{\infty,\gamma},\,&0\le s\le t)H_2((\sum_{k=[\frac{t}{\gamma}]+1}^{[\frac{t+s}{\gamma}]}\xi_k)_{s\in[0,T]})]\\&
=\ES[H_1(\bar{X}_s^{\infty,\gamma},\,0\le s\le t)]\ES[H_2((\sum_{k=[\frac{t}{\gamma}]+1}^{[\frac{t+s}{\gamma}]}\xi_k)_{s\in[0,T]})].
\end{align*}
The sequel of the proof follows the lines of that of Proposition \ref{innovation-discrete}. We leave it to the reader.
\end{proof}
\section{Appendix}
In the Appendix we give the proof of two technical results.
\label{appendix}
\begin{prop} \label{propcontrolesup}Assume that $(Z_t)_{t\ge0}$ satisfies $\mathbf{(H_1)}$. Then, for every $T>0$, for every $r>0$, $\ES[\sup_{t\in[0,T]}|Z_t|^r]<+\infty$. Moreover, there exist $p\ge1$ and $T_0>0$ such that for every
$T\le T_0$,
\begin{equation*}
\ES[\sup_{t\in[0,T]}|Z_t|^{p}]\le C T^{1+\rho}\quad\textnormal{with $\rho>0$.}
\end{equation*} 
\end{prop}
\begin{proof}
First, note that it is enough to prove the result for every coordinate $Z^j$ with $j\in\{1,\ldots,\ell\}$. Therefore, it is in fact enough to prove that
the results are true for  any one-dimensional centered Gaussian process with stationary increments and variance function $(c(t))_{t\ge0}$  satisfying $\mathbf{(H_1)}$. Then, for every $t>0$ and $\varepsilon>0$,  $c(t)=\ES[Z_t^2]$ 
 and denote by ${\cal D}(T,\varepsilon)$ the Dudley integral defined by
$${\cal D}(T,\varepsilon)=\int_0^\varepsilon (\log(N(T,u))^{1/2}du,$$
where, for $ u >0,$ 
$$N(T,u)=\inf\{n\ge1, \exists \; s_1,\ldots, s_n \textnormal{ such that $\forall t\in [0,T]$, $\exists i\in\{1,\ldots,n\}$ with $\sqrt{c(t-s_i)}\le u$}\}.$$
By the Dudley Theorem (see $e.g.$ Theorem 1 of \cite{lifshits} p 179), for every $T>0$,
\begin{equation}\label{dudley}
\ES[\sup_{t\in[0,T]}|Z_t|] \le 2 \ES[\sup_{t\in[0,T]}Z_t]\le C{\cal D}(T,\sqrt{\bar{c}(T)})\quad\textnormal{with }\bar{c}(T)=\sup_{t\in[0,T]}c(t).
\end{equation} 
Let us control the right-hand member. By assumption $\mathbf{(H_1)}$ and \eqref{consequence-cov-bounds}, for every $\delta\in(0,1)$,
$$ \bar{c}(\delta)\le C \delta^{\mu}\quad\textnormal{where $C$ does not depend on $\delta$ and}$$
$\mu\in(0,1]$ (depending on the value of $a_1$). It follows that, for
$u >0,$  
$$N(T,u)\le C T u^{2/\mu},$$
where $C$ does not depend on $T$. 
For  $\varepsilon>0$ small enough, 
\begin{align*}
{\cal D}(T,\varepsilon)&\le \int_0^\varepsilon  |\log(CT) + \frac2{\mu} \log(u)|^{1/2} du
\le C \varepsilon |\log(\varepsilon)|^{1/2}.
\end{align*}
It follows from \eqref{dudley} that there exists $T_0 >0$ such that for $ T \le T_0$
\begin{align*}
\ES[\sup_{t\in[0,T]}|Z_t|]&\le C (\bar{c}(T))^{1/2}|\log(\bar{c}(T))|^{1/2}\\
                          &\le C T^{\mu/2}|\log(T)|^{1/2}.
\end{align*}
Then by Corollary 3.2 in~\cite{ledouxtal} $ \ES[\sup_{t\in[0,T]}|Z_t|^r] < + \infty $ 
for every $T >0 $ and every $r >0$ and 
$$ \ES[\sup_{t\in[0,T]}|Z_t|^p]\le C T^{\mu p/2}|\log(T)|^{p/2},$$
for  $ T \le  T_0.$ One can choose $p$ big enough to prove the second inequality
in the Proposition. 
\end{proof}
 In this last part of the Appendix, we specify the constructions of the stationary solutions when $(\Delta_n)_{n\ge1}$ and $(Z_t)_{t\ge0}$ admit the representations \eqref{rep-discrete} and \eqref{rep-movingaverage} respectively. 
Let us first denote by $P_{\xi}$ the distribution of the innovation sequence $  \xi=(\xi_n)_{n\in\mathbb{Z}}$ 
and by $P_{ \bf W}$ the Wiener measure where $  {\mathbf W}= (W_t)_{t\in\ER} $ is a two-sided 
 standard Brownian motion.

 Let $\kappa_1:\ER\times\ER^\mathbb{Z}\rightarrow\ER^{\mathbb{N}}$ and 
$\kappa_2:\ER\times{\cal C}(\ER,\ER)\rightarrow{\cal C}(\ER_+,\ER)$ be respectively 
defined by:
$(\kappa_1(x,(\xi_n)_{n\in\mathbb{Z}}))_0=x$, $(\kappa_2(x,(W_t)_{t\in\ER})_0=x$, 
\begin{equation}
\label{kappa1}
 (\kappa_1(x,(\xi_n)_{n\in\mathbb{Z}}))_{n+1}=(\kappa_1(x, \xi))_n+\gamma b((\kappa_1(x,\xi))_n)+\sum_{k=0}^{+\infty} a_{k,\gamma} \xi_{n-k} \quad \textnormal{for every $n\ge0$}, 
 \end{equation}
 and,
 \begin{equation}
\label{kappa2}
(\kappa_2(x,(W_t)_{t\in\ER})_t=x+\int_0^t b(\kappa_2(x,{\bf W}))_sds+\int_{-\infty}^t f_t(s)dW_s.
\end{equation}
Note that $\kappa_2$ is well-defined $ \mu-a.s.$ for all measure $\mu $ such that $ P_2\mu_2 =P_{ \bf W}$ 
(where $P_2$ denotes the projection of the probability on the second coordinate)
 since $b$ is a Lipschitz continuous function. We are now ready to state that there exist $\mu_1$ and $\mu_2$ 
such that the stationary solutions to the discretized and continuous equations (when they exist) can be realized on $(\ER\times\ER^{\mathbb{Z}},\mu_1)$ and $(\ER\times{\cal C}(\ER,\ER),\mu_2)$ respectively.
\begin{lemme} \label{lemme-verif-const} (i) Assume that $(\Delta_n)_{n\ge1}$ admits the representation \eqref{rep-discrete}. Let ${\cal P}$ be a probability on $\ER^{\mathbb{N}}$
that denotes the  distribution of a stationary solution to the discrete recursive equation 
$$\bar{X}_{(n+1)\gamma} -\bar{X}_{n\gamma}=\gamma b(\bar{X}_{n\gamma})+\Delta_{n+1}\quad \forall n\ge0.$$
Then, there exists a probability $\mu_1$ on $\ER\times\ER^{\mathbb{Z}}$ with $P_2\mu_1=P_{  \xi} $ such that  $\kappa_1(\mu_1)= \cal P.$\\
(ii) Assume that $(Z_t)_{t\ge0}$ admits the representation \eqref{rep-movingaverage}. Let ${\cal P}$ be a probability on ${\cal C}(\ER_+,\ER)$ that denotes the distribution of a stationary solution to \eqref{fractionalSDE}. Then, there exists a probability $\mu_2$ on $\ER\times{\cal C}(\ER,\ER)$ with $P_2\mu_2= P_{ \bf W}$ such that $\kappa_2(\mu_2)$ is a stationary solution to \eqref{fractionalSDE} with distribution ${\cal P}$.
\end{lemme}
\begin{proof} (i) First, let $\phi_1:\ER^{\mathbb{N}}\rightarrow\ER\times\ER^{\mathbb{N}^*}$ be defined by
$$\phi_1((x_n)_{n\ge0})=(x_0,(\delta_n)_{n\ge1})\quad \textnormal{with}\quad \delta_n:=x_{n}-x_{n-1}-\gamma b(x_{n-1}).$$
By construction $\nu_1:=\phi_1({\cal P})$ defines a probability on $\ER\times\ER^{\mathbb{N}^*}$ with $P_2\nu_1={\cal L}((\Delta_n)_{n\ge1})$ and such that if $(X_0,(\Delta_n)_{n\ge1})$ has distribution $\nu_1$, then the induced Euler scheme has distribution ${\cal P}$. Let 
$\pi_1(.,dx)$ denotes the conditional distribution of $X_0$ given $(\Delta_n)_{n\ge1}$. Let ${\cal R}:\ER^{\mathbb{Z}}\rightarrow\ER^{\mathbb{N}^*}$ denote the function defined by ${\cal R}((\xi_n)_{n\in\mathbb{Z}})=(\sum_{k=0}^{+\infty} a_{k,\gamma} \xi_{n-k})_{n\ge1}$. Then, define $\mu_1$ on $\ER\times\ER^{\mathbb{Z}}$ by
\begin{equation}
  \label{eq:mu1}
 \mu_1(dx,d \xi)=\pi_1({\cal R}( \xi,dx))\otimes P_{  \xi}(d  \xi).
\end{equation}

Setting $\tilde{\cal R}(x, \xi)=(x,{\cal R}( \xi))$, it follows from the very definition of $\mu_1$ that $\tilde{\cal R}(\mu_1)=\nu_1$. 
As a consequence, $\kappa_1(\mu_1)={\cal P}$. This completes the proof.\\
(ii) The proof of the second point is similar and is left to the reader.
 \end{proof}


\end{document}